\documentclass[12p]{article}
\usepackage{amsfonts,amssymb,amsmath}
\usepackage{lscape}
\newcommand{\GF}{\mathop{\mathrm{GF}}}
\newcommand{\GL}{\mathop{\mathrm{GL}}}
\newcommand{\SL}{\mathop{\mathrm{SL}}}
\newcommand{\PGL}{\mathop{\mathrm{PGL}}}
\newcommand{\PSL}{\mathop{\mathrm{PSL}}}
\newcommand{\tr}{\mathop{\mathrm{tr}}}

\newtheorem{theorem}{Theorem}[section]
\newtheorem{lemma}[theorem]{Lemma}
\newtheorem{fact}[theorem]{\relax}
\newtheorem{unnumber}{}

\newtheorem{prepf}[unnumber]{Proof}
\newenvironment{proof}{\prepf\normalfont}{\endprepf}
\newtheorem{prem}[theorem]{Remark}
\newenvironment{remark}{\prem\normalfont}{\endprem}

\renewcommand{\wr}{\mathbin{\mathrm{wr}}}

\begin{document}

\title{Primitive permutation groups of degree $3p$}
\author{by Peter M. Neumann}
\date{}
\maketitle

This paper presents an analysis of primitive permutation groups of degree~$3p$,
where $p$ is a prime number, analogous to H.~Wielandt's treatment~\cite{r19}
of groups of degree~$2p$. It is also intended as an example of the systematic
use of combinatorial methods as surveyed in \S6 for distilling information
about a permutation group from knowledge of the decomposition of its
character. The work is organised into three parts. Part~I contains the
lesser half of the calculation, the determination of the decomposition of the
permutation character. Part~II contains a survey of the combinatorial methods
and, based on these methods, the major part of the calculation. Part~III
ties up loose ends left earlier in the paper and gives a tabulation of detailed
numerical results.

\begin{center}\Large\textbf{Part I}\end{center}

\section{Introduction}

The results can be summarised as follows: if $p$ is prime and $p\ge5$, then 
every primitive permutation group of degree~$3p$ is $2$-fold transitive
unless, for some integer $a$,
\begin{center}
\begin{tabular}{lrlll}
   & (i)   & $p=3a^2+3a+1$    & ($a\ge1$) \\
or & (ii)  & $p=48a^2+30a+5$  & ($a\ge0$) \\
or & (iii) & $p=48a^2+66a+23$ & ($a\ge0$) \\
or & (iv)  & $p=7$, $19$ or $31$. & \\
\end{tabular}
\label{p2}
\end{center}

It will be seen later why cases (i) and (iv) are allowed to overlap.

In cases (ii) and (iii) simply transitive primitive groups must have rank~$3$
(in the sense of D.~G.~Higman~\cite{r9}), and in all cases the rank of our
group is at most~$4$. The calculation gives the lengths of the suborbits, and
other numerical information of a similar kind. Details are tabulated at the
end of the paper.

Simply transitive, primitive groups of degree~$3p$ do exist for $p=5$,
$7$ and $19$:

\paragraph{Example 1.} $A_6$, $S_6$ operating on unordered pairs are 
primitive groups of degree~$15$ and rank~$3$ (case~(ii) with $a=0$).

\paragraph{Example 2.} $A_7$, $S_7$ operating on unordered pairs are
primitive groups of degree~$21$ and rank~$3$ (case~(i) with $a=1$).

\paragraph{Example 3.} $\PGL(2,7)$ operating by conjugation on its
set of Sylow $2$-subgroups is a primitive group of degree~$21$ and rank~$4$,
exemplifying case~(iv).\label{e:ex3}

\paragraph{Example 4.} $\PSL(2,19)$ operating by right multiplication
on the cosets of one of its subgroups $A_5$ is primitive of degree~$57$ and
rank~$4$, again an instance of case~(iv).

\medskip

How these four examples fit into the more detailed scheme of Table~\ref{t52} is
shown in Table 14.1. I have no idea whether there are more. The possibility
$p=31$ in case (iv) is of particular interest, and I hope to show in a later
paper how the combinatorial data collected in this work can be used to decide
whether or not it does actually correspond to a group.

The context of the results is the problem of classifying as far as is possible
the primitive permutation groups of degree~$kp$ where $l$ is small (compared
with~$p$). I had hoped that the methods I use would extend to higher values
of $k$, but even for $k=4$ the complexity of the system of simultaneous
diophantine equations which arises seems to become prohibitive. Nevertheless,
if one is concerned to test only relatively weak hypotheses, then the methods
may just be sufficient, for example,
\begin{quote}
(A) for each value of $k$ ($k\ge2$) there is a finite set $p_{k,1}(a)$, \dots,
$p_{k,f(k)}(a)$ of quadratic polynomials with integer coefficients such that,
if $p\ne p_{k,i}(a)$ for some integers $i,a$, then all primitive groups of
degree $kp$ are $2$-fold transitive;
\end{quote}
or the weaker
\begin{quote}
(B) for each value of $k$ ($j\ge2$) the set $P_k$ defined by
\begin{eqnarray*}
P_k &=& \{p\mid p \hbox{ prime and there exists a simply transitive, primitive}\\
&& \qquad\qquad \hbox{ permutation group of degree $kp$\}}
\end{eqnarray*}
is no more ``dense'' than the set of all perfect squares.
\end{quote}

That (A) is true for $k=2$ is proved by Wielandt~\cite{r19,r20}; and for
$k=3$ it is still quite easy to show: much of the detailed calculation of
\S\S7--10 goes to prove the refined version stated in the first paragraph of
this paper. And even before I had been able to complete the proof of
Theorem 9.2, B.~J.~Birch had shown me how to prove (B) (see \S13).

Of course, this is still a very far cry from the classification of the groups
themselves. Partial results along these lines have been obtained by
Ito~\cite{r12,r13} and Rowlinson~\cite{r15} only for $k=2,3,4$ under very
restrivtive conditions on the normalisers of Sylow $p$-subgroups.

The main part of the calculation is based on the language and notations
introduced in \S6. Little of this is new, except perhaps the point of view,
but it unifies and simplifies various pieves of information scattered
throughout the literature. Apart from this the work contains no new ideas
and the proof is similar to Wielandt's proof for groups of
degree~$2p$~\cite{r19,r20}. The difference lies in the fact that, for groups
of degree~$3p$, even once the decomposition of the permutation character is
known, there is still a great deal of calculation to be done, far heavier
than the corresponding calculation for groups of degree $2p$
(cf.~\cite{r20}, pp.102, 103). I hope that in spite of the weight of the
calculations, the proof will support my thesis that the point of view
proposed in \S6 is a sensible one.

In the unpublished part of his thesis, which he kindly made available to me,
Tamaschke has done a great deal of work on groups of degree~$3p$, but apart
from the initial reductions of the problem his calculations proceed along
quite different lines from mine. L.~L.~Scott~\cite{r16} has also found
partial results -- but I do not have details yet.

It is a pleasure to acknowledge with thanks the help of Mr~J.~E.~Stoy, whose
calculations on the Oxford University computer showed me how to complete the
proof of Theorem~9.2; and of Mr~P.~Rowlinson who read the first draft of this
paper and helped with discussion and informed comment.

\section{Notation}

Throughout the paper $p$ is prime, $G$ is a primitive but not $2$-fold 
transitive permutation group of degree $3p$, $H$ denotes a stabiliser in $G$,
and $P$ a Sylow $p$-subgroup of $G$. The \emph{suborbits} of $G$ are the 
orbits of $H$, the \emph{subdegrees} are the lengths of the suborbits, and
the \emph{rank} $r$ of $G$ is the number of suborbits.

We will identify $G$ with the linear group given by the appropriate permutation
matrices. The field over which this linear group is defined will depend on the
context -- usually it will be implicitly assumed to be a suitable cyclotomic
extension of the rational field $\mathbb{Q}$. The principal character of a
group $X$ will be denoted $1_X$. We use $F$ to denote the cyclotomic field
$\mathbb{Q}(\theta)$, where $\theta$ is a primitive $|G|$-root of $1$. It
will be proved later, and assumed now, that $|G|=p.m$ where $p\nmid m$. Thus
$F$ is generated by a primitive $p$-th root $\theta^m$ of $1$ and a primitive
$m$-th root $\theta^p$. Let $\Gamma$ be the automorphism group of $F$ (over
$\mathbb{Q}$) and $\Gamma^*\le\Gamma$ the subgroup fixing $\theta^p$. Thus
$\Gamma^*$ is cyclic of order $p-1$, and permutes the primitive $p$-th
roots of unity in $F$ regularly. Every complex character of $G$ has its
values in $F$, and so $\Gamma$ operates on this set: if $\chi$ is a
character of $G$ and $\gamma\in\Gamma$ we define $\chi^\gamma$ by
\[\chi^\gamma(g)=\chi(g)\gamma\]
for all $g\in G$. It is well known that $\chi^\gamma$ is again a character.
The character $\chi$ is said to be \emph{$p$-rational} if $\chi^\gamma=\chi$
for all $\gamma\in\Gamma^*$.

The names used for certain special groups are more or less standard:
$S_n$, $A_n$ for the symmetric and alternating groups on the set
$\{1,2,\ldots,n\}$; $\GL(m,q)$, $\SL(m,q)$ for the general and special
(unimodular) linear groups respectively over the field $\GF(q)$ with $q$ 
elements; $\PGL(m,q)$, $\PSL(m,q)$ for the corresponding projective groups.
Other notation and terminology is explained in \S6.

\section{Preliminary lemmas}

Primitive permutation groups of degrees $6$ and $9$ are known. In fact those
of degree~$6$ are all $2$-fold transitive: $S_6$, $A_6$, $\PGL(2,5)$,
$\PSL(2,5)$, the last two operating on the points of the projective line over
$\GF(5)$. If $G$ is primitive of degree~$9$, then by a theorem of 
Jordan~\cite[p.32]{r20}, the only primes dividing $|G|$ are $2$ and $3$.
Consequently $G$ is soluble, and contains a regular elementary abelian normal
subgroup of order~$9$. The stabiliser $H$ is a complement for $T$ in $G$, and,
as it is represented faithfully as a group of automorphisms of $G$, it can be
identified as a subgroup of $\GL(2,3)$. There are just two such primitive
groups which are not $2$-fold transitive: a Frobenius group in which $H$,
the Frobenius complement, is cyclic of order~$4$; and a group in which $H$
is dihedral of order~$8$. Both these groups are of rank $3$ and have
subdegrees $1,4,4$. \textit{From now on we will suppose that $p\ge5$.}

\begin{lemma}
$p^2\nmid |G|$.
\label{l31}
\end{lemma}

\begin{proof}
If $p^2\mid|G|$ then certainly $p\mid|H|$ and $G$ would contain an element
$g$ of order $p$ having at least one fixed point. Then $g$ has degree $p$
(and $2p$ fixed point) or degree $2p$ (and $p$ fixed points) and, by
theorems of Jordan and Manning (see \cite{r20},~p.39) $G$ would be alternating
or symmetric of degree~$3p$. This contradicts our assumption that $G$ is not
$2$-fold transitive, and proves the lemma: and the Sylow $p$-subgroup $P$
is cyclic of order~$p$.
\end{proof}

Next a triviality about the normal structure of $G$:

\begin{lemma}
$G$ contains no non-trivial abelian normal subgroup.
\label{l32}
\end{lemma}

\begin{proof}
If $K$ were a non-trivial abelian normal subgroup then since $G$ is 
primitive $K$ would be transitive, hence regular. But then the $p$-primary
constituent of $K$ would be normal in $G$ but intransitive, and $G$ coud
not be primitive.
\end{proof}

\begin{remark}
Tamaschke observes in \cite{r18} that in fact any non-trivial normal subgroup
in $G$ is again primitive except in case $p=7$, $G=\PGL(2,7)$ as
in Example~3 (on page~\pageref{e:ex3}). This is a genuine exception, for
$\PSL(2,7)$ operating as a transitive group of degree~$21$ is
imprimitive (it is a group of rank~$6$ and (see below) Type VIII, with
subdegrees $1,2,2,4,4,8$).
\end{remark}

Thirdly, a well-known fact about the permutation character of $G$.

\begin{lemma}
Every non-principal constituent of $\pi$ is faithful.
\label{l34}
\end{lemma}

\begin{proof} Let $\chi$ e a non-principal character of $G$ which appears as
a constituent of $\pi$, and let $K$ be the kernel of $\chi$. Then certainly
$K$ is a proper normal subgroup of $G$. The restriction $\pi|_K$, which is the
appropriate permutation character of $K$, contains $(1+\chi)|_K$ which is a
proper multiple of the principal character of $K$. Consequently $K$ is
intransitive on $\Omega$, and since $K$ is normal, the orbits of $K$ form a
system of blocks for $G$. As $G$ is primitive each of these blocks contains
just $1$ element, and so $K=1$. That is, $\chi$ is a faithful character.
\end{proof}

\section{A theorem of Feit and the groups $\PSL(2,q)$}

Our aim in this section is to prove the crucial

\begin{lemma}
If $f$ is the degree of a non-principal irreducible constituent of $\pi$,
then $f\ge p-1$.
\label{l41}
\end{lemma}

This depends on a deep theorem of Feit~\cite{r3} who proved that a finite group
whose centre is trivial and which has a faithful character of degree less
than $p-1$ either has a normal Sylow $p$-subgroup or is isomorphic to
$\PSL(2,p)$ or to $\SL(2,q)$ where $q$ is a power of $2$
and $p=q+1$. In case $\pi$ did contain a non-principal constituent of degree
less than $p-1$, by Lemmas~\ref{l32} and~\ref{l34} Feit's theorem would be
applicable, and since by (\ref{l32}) our Sylow $p$-subgroup is not normal,
we would have either $G=\PSL(2,p)$ or $G=\SL(2,q)$. The proof
is therefore completed by the following two lemmas.

\begin{lemma}
If $q$ is a power of $2$, $p=q+1$ and $p\ge5$, then $\SL(2,q)$
(i.e.~$\PSL(2,q)$) has just one representation as a transitive
permutation group of degree $3p$, but this is imprimitive.
\label{l42}
\end{lemma}

\paragraph{N.B.} In fact the permutation character is $1+2\chi_1+\chi_2$ 
where $\chi_1$ is irreducible of degree~$p$, $\chi_2$ is irreducible of
degree~$p-1$. Thus, even as an imprimitive group this is not a
counter-example to Lemma~\ref{l41} (see \S11).

\begin{proof}
Let $H$ be the stabilizer in a representation of $\SL(2,q)$ as a
transitive group of degree~$3p$. Since the index of $H$ is odd, $H$ contains
a Sylow $2$-subgroup $U$ of $\SL(2,q)$. If $U\ntriangleleft H$ then $H$
contains at least one further Sylow $2$-subgroup, but since $U$ permutes the
other Sylow $2$-subgroups of $\SL(2,q)$ transitively by conjugation, it would
follow that $H$ contained all Sylow $2$-subgroups and would be the whole
group. Hence $U\lhd H$, and $H\le N(U)$. Now $N(U)$ has index $p$ in 
$\SL(2,q)$ and so $|N(U):H|=3$. In fact, $N(U)$ is the split extension of $U$
by a cyclic group of order $q-1$. SInce $q$ and $q+1$ are not multiples
of $3$ it follows that $3$ does divide $p-1$, therefore $N(U)$ contains
precisely one subgroup of index~$3$, and the lemma is proved.
\end{proof}

\begin{lemma}
If $\PSL(2,p)$ has a representation as a transitive permutation group of
degree~$3p$ then $p$ is $5$, $7$ or $19$. In case $p$ is $5$ or $7$ the
groups are imprimitive. $\PSL(2,19)$ can be represented in two (similar)
ways as a primitive group of degree $3.19$; the stabilizer is isomorphic
to $A_5$; the rank is $4$; the subdegrees are $1$, $6$, $20$, $30$; the
degrees\footnote{This group is a counterexample of rank~$4$ to a conjecture
of Frame, conjecture (B) on p.89 of \cite{r20}; it gives a negative answer
to the implicit question raised in \cite{r20}, p.93, 11.3,4. Compare the
footnote on p.\pageref{p46}, and Ito~\cite{r14}.} of the irreducible
constituents of the permutation character are $1$, $18$, $18$, $20$.
\label{l43}
\end{lemma}

\begin{proof}
We take $G$ to be $\PSL(2,p)$. and hunt for the subgroup $H$ of index~$3p$
by comparing with the list of all subgroups in $\PSL(2,p)$ given, for
example, in Burnside~\cite{r1}. We have the possibilities
\begin{itemize}
\item[(i)] $H$ is contained in a dihedral group of order $p-1$;
\item[or (ii)] $H$ is contained in a dihedral subgroup of order $p+1$;
\item[or (iii)] $H$ is contained in a maximal subgroup isomorphic to $A_4$,
$S_4$ or $A_5$.
\end{itemize}
Possibility (i) gives that $\frac{1}{2}p(p+1)$ divides $3p$, whence $p=5$;
case (ii) can only arise if $\frac{1}{2}p(p-1)$ divides $3p$, that is, if
$p=7$; and in case (iii), if $H$ is a proper subgroup of the relevant maximal
subgroup, then either $H$ is isomorphic to $A_4$ or it is aready covered by
cases (i) and (ii). Thus we may assume that $H$ has order $12$, $24$ or $60$.
These give
\begin{eqnarray*}
3p.12 &=& \frac{1}{2}p(p^2-1)\\
3p.24 &=& \frac{1}{2}p(p^2-1)\\
3p.60 &=& \frac{1}{2}p(p^2-1)\\
\end{eqnarray*}
respectively. Of these only the last is possible and it gives $p=19$.

Since $19\equiv-1\pmod{5}$, $\PSL(2,19)$ contains two conjugate classes of
maximal subgroups isomorphic to $A_5$ (cf.~Burnside~\cite{r1}, \S324),
and therefore $\PSL(2,19)$ has two representations as a primitive group of
degree $3.19$. The two classes of $A_5$-subgroups are interchanged by the
outer automorphism of $\PSL(2,19)$, and our two primitive groups of degree
$3.19$ are therefore similar.

A straightforward calculation from the character table of $\PSL(2,19)$ yields
that the degrees of the irreducible constituents of the permutation character
are $1$, $18$, $18$, $20$. Knowing that the rank therefore is $4$ and
knowing that the stabilizer is $A_5$, it is easy to compute that the
subdegrees must be $1$, $6$, $20$, $30$ (compare \S10). In any case, since
there is no constituent of degree less than $p-1$ in the permutation
character of this group, it follows that Lemma~\ref{l41} is correct.
\end{proof}

\paragraph{N.B.} Even as imprimitive groups of degrees $15$, $21$ respectively
$\PSL(2,5)$ and $\PSL(2,7)$ are not counterexamples to Lemma~\ref{l41}.
For in the case of $\PSL(2,5)$ one finds that $\pi=1+2\chi_1+\chi_2$ where
$\chi_1$ is irreducible of degree~$5$. $\chi_2$ irreducible of degree~$4$;
and in the case of $\PSL(2,7)$, $\pi=1+\chi_1+2\chi_2$ where $\chi_1$ is
irreducible of degree~$8$, $\chi_2$ irreducible of degree~$6$.

\section{Reduction of the permutation character of $G$}

Suppose that $\pi=1_G+\sum e_i\chi_i$ where the characters $\chi_i$ are
distinct non-principal irreducible characters of $G$, Since $\pi|_P=3\rho_P$,
where $\rho_P$ is the character of the regular representation of $P$, we have
\[\left(\sum e_i\chi_i\right)|_P=2.1_P+3\sum_{\lambda\in\Lambda^*}\lambda\]
where $\Lambda^*$ is the set of non-principal linear characters of $P$. We may
assume the numbering to be chosen so that, for $1\le i\le t$, $\chi_i|_P$
contains $1_P$, and for $i>t$, $\chi_i|_P$ does not contain $1_P$. Clearly
then $t\le2$. The operation of the Galois group $\Gamma$ leaves $\pi$
invariant and therefore permutes the irreducible constituents of $\pi$:
the characters $\chi_i$, $i\le t$, are permuted among themselves, and the
characters $\chi_i$, $i>t$, are permuted among themselves. Thus if
$\chi_a=\sum_{i\le t}e_i\chi_i$ and $\chi_b=\sum_{i>t}e_i\chi_i$ then
$\chi_a$ is a rational character and $\chi_b$ is either $0$ or a rational
character of $G$. Furthermore $\chi_a$ is faithful (Lemma~\ref{l34}) and
so $\chi_a|_P$ contains the non-principal linear characters of $P$ each at
least once and all with the same multiplicity. Similarly, $\chi_b$ is either
$0$, or of degree~$p-1$, or of degree~$2(p-1)$. We take these possibilities
in turn.

Suppose first that $\chi_b=0$. Since $G$ is not $2$-fold transitive $\chi_a$
cannot be irreducible, and therefore is the sum of two irreducible
constituents, $\chi_1$ and $\chi_2$. (We know that $\chi_1\ne\chi_2$ because
$\chi_1|_P$ is not the double of any character of $P$.) If $\chi_1$ is not
$p$-rational then $\chi_1$ and $\chi_2$ must be interchanged by $\Gamma^*$.
In this case therefore $\chi_1$ and $\chi_2$ have the same degree, 
$(3p-1)/2$, and if $\Gamma^{**}$ denotes the (unique) subgroup of index~$2$
in $\Gamma^*$, then
\[\chi_i|_P=\rho_P+\sum\{\lambda\mid\lambda\in\Lambda_i^*\}\]
where $\Lambda_1^*$, $\Lambda_2^*$ are the two orbits of $\Gamma^{**}$
operating on $\Lambda^*$. This is \textit{case I}. If $\chi_1$ is $p$-rational
then so is $\chi_2$, and by suitable choice of the labels we have 
\textit{case II}\/: $\deg\chi_1=p$, $\deg\chi_2=2p-1$.

Next suppose that $\deg\chi_b=p-1$. By Lemma~\ref{l41} the constituents of
$\chi_1$ have degree at least $p-1$, and in fact it is clear that $\chi_1$
cannot have a constituent of degree~$p-1$. Thus there are just three 
possibilities: either $\chi_a$ is irreducible, or $\chi_1=\chi_1+\chi_2$
where $\chi_1=\chi_2$ and $\chi_1,\chi_2$ both have degree~$p$, or
$\chi_1=2\chi_1$ where $\chi_1$ has degree $p$. These are \textit{cases III,
IV, V} respectively.

Finally, if $\deg\chi_b=2(p-1)$, then $\chi_a$ has degree~$p+1$ and,
by Lemma~\ref{l41}, must be irreducible. Again by \ref{l41}, $\chi_b$ has at
most two irreducible constituents, and if it does have two, then they both
have degree~$p-1$. The three possibilities, $\chi_b$ irreducible of degree
$2(p-1)$, $\chi_b$ the sum of two distinct irreducible characters each of
degree~$p-1$, $\chi_b$ twice an irreducible character of degree~$p-1$, are
\textit{cases VI, VII, VIII} respectively in Table~\ref{t52}.

Further information about the irreducible characters which appear in the
decomposition of $\pi$, for example that even in cases IV and VII the
characters of degree $p$, $p-1$ respectively are $p$-rational, can be obtained
directly from Suzuki's theory of exceptional characters or (equivalently in
this case) from Brauer's theory of modular characters. However, I have not
been able to make use of such facts (though for groups of degree~$4p$ the
analogous information does seem to be useful -- see \cite{r15}), and details
are left to the reader. Only in case I do we need a small amount of additional
information:

\begin{lemma}
Suppose that $\pi=1+\chi_1+\chi_2$ where $\chi_1$ and $\chi_2$ are
irreducible characters of degree $(3p-1)/2$. Then
\begin{itemize}
\item[(i)] $\overline{\chi_1}=\chi_2$ if $p\equiv3\pmod{4}$;
\item[(ii)] $\overline{\chi_1}=\chi_1$ if $p\equiv1\pmod{4}$.
\end{itemize}
\label{l51}
\end{lemma}

\begin{proof}
Complex conjugation operates on $\{\chi_1,\chi_2\}$ in the same way as the
(unique) element of order~$2$ in $\Gamma^*$. Since $\Gamma^*$ has order
$p-1$, if $p\equiv3\pmod{4}$ then $\Gamma^{**}$ has odd order and complex
conjugation interchanges $\chi_1$ and $\chi_2$; while if $p\equiv1\pmod{4}$
then $\Gamma^{**}$ contains the element of order~$2$ in $\Gamma^*$ and
conjugation leaves $\chi_1$ and $\chi_2$ unchanged.
\end{proof}

The entries in the main part of the following table are the degrees $f_i$ of
the irreducible constituents $\chi_i$ of $\pi$. All multiplicities $e_i$
are $1$ except in the two cases shown. $\chi_0$ is the principal character.

\setcounter{table}{1}
\begin{table}[htbp]
\[\begin{array}{|cccccc|}
\hline
\hbox{Type} & r & \chi_0 & \chi_1 & \chi_2 & \chi_3 \\
\hline
\hbox{I} & 3 & 1 & (3p-1)/2 & (3p-1)/2 & \\
\hbox{II} & 3 & 1 & p & 2p-1 & \\
\hbox{III} & 3 & 1 & 2p & p-1 & \\
\hbox{IV} & 4 & 1 & p & p & p-1 \\
\hbox{V} & 6 & 1 & p & p-1 & \\
&&&\hbox{(mult.~$2$)} & & \\
\hbox{VI} & 3 & 1 & p+1 & 2p-2 & \\
\hbox{VII} & 4 & 1 & p+1 & p-1 & p-1 \\
\hbox{VIII} & 6 & 1 & p+1 & p-1 & \\
&&&&\hbox{(mult.~$2$)} & \\
\hline
\end{array}\]
\caption{\label{t52}the decomposition of $\pi$}
\end{table}

The second part of this paper is devoted to showing that cases I, V, VI, VIII
cannot arise -- thus, in particular, the rank of $G$ is at most $4$; that
case II corresponds to possibilities (ii), (iii) of page~\pageref{p2}; that
cases III, IV give possibility (i) of page~\pageref{p2}; and that case VII
cannot arise unless $p$ is $7$, $19$ or $31$.

\begin{center}\Large\textbf{Part II}\end{center}

\section{Survey of combinatorial methods}

Suppose now for convenience that $\Omega$, the set on which $G$ operates, is
\hfil\break
$\{0,1,\ldots,n-1\}$ (in our case $n=3p$), and that $H=G_0$. If
$\Delta_0,\Delta_1,\ldots,\Delta_{r-1}$ are the orbits of $G$ operating in
the natural way on $\Omega\times\Omega$, we define the corresponding
\emph{basic adjacency matrices} $B_0,\ldots,B_{r-1}$ by
\[\left((B_i)_{\alpha\beta}\right)=
\begin{cases}
1 & \hbox{if }(\alpha,\beta)\in\Delta_i\\
0 & \hbox{if }(\alpha,\beta)\notin\Delta_i\\
\end{cases}.\]
Then (cf.~Wielandt~\cite{r20}, p.80) the matrices $B_0,\ldots,B_{r-1}$ span
the algebra $V$ of all matrices which commute with the permutation matrices
representing $G$. The orbits $\Delta_i$ correspond in a natural way to the
orbits of $H$ in $\Omega$; the length of the corresponding $H$-orbit can be
read off as the number of entries equal to $1$ in each row (or column) of
$B_i$, that is, the subdegree is the row (or column) sum of $B_i$; the matrix
$B_i$ is the transpose of $B_j$ if and only if the corresponding orbits are
paired (\cite{r20} p.83), and in particular, self-paired orbits correspond to
symmetric basic adjacency matrices.

Since $G$ is transitive on $\Omega$ the diagonl is one $G$-orbit in
$\Omega\times\Omega$, and if we choose the numbering correctly we can take it
to be $\Delta_0$: the corresponding suborbit is $\{0\}$, and $B_0$ is the
identity matrix. We write $\Delta^*$ for the orbit of $G$ in 
$\Omega\times\Omega$ obtained by reversing all the pairs in $\Delta$:
\[\Delta^*=\{(\alpha,\beta)\mid(\beta,\alpha)\in\Omega\}.\]
Then $\Delta^*$ and $\Delta$ correspond to paired suborbits of $G$ in $\Omega$.
We will also use $*$ to denote an involution of $\{1,\ldots,r-1\}$ in such
a way that
\[\Delta_i^*=\Delta_{i^*}.\]
Thus $B_{i^*}$ is the transpose of $B_i$.

There is a convenient and suggestive geometrical interpretation of the orbits
$\Delta_i$. As Sims has pointed out~\cite{r17}, if $i\ne0$, the orbit
$\Delta_i$ is a directed graph with vertex set $\Omega$, and $G$ is a group
of graph automorphisms. As a graph $\Delta_i$ has the special property that
every point has edges emanating from it, and furthermore $G$ is transitive
on the directed edges. We will say that $\Delta_i$ admits $G$ as a
\emph{flag-transitive} automorphism group. Since we are interested not only
in each individual suborbit, that is, each individual graph, but also in the 
way that these graphs fit together to form the complete directed graph
$(\Omega\times\Omega)-\Delta_0$, it is appropriate to use the notions of
edge-coloured graphs. We choose colours $c_1,\ldots,c_{r-1}$ and colour the
edge $(\alpha,\beta)$ with colour $c_i$ if and only if
$(\alpha,\beta)\in\Delta_i$, to make a \emph{coloured directed graph}
$C(G,\Omega)$. This ensures of course that the \emph{monochrome subgraphs}
-- i.e.\ the subgraphs defined by all the vertices and \emph{all} the edges
of one given colour -- are just the original graphs
$\Delta_1,\ldots,\Delta_{r-1}$. As an edge-coloured directed graph
$C(G,\Omega)$ has several special properties. For example, its automorphism
group contains $G$ and is therefore transitive on the flags of each colour;
and if $\Delta$ is a monochrome subgraph, then so is $\Delta^*$, the graph
obtained by reversing all edges of $\Delta$. If all the suborbits of $G$
are self-paired then each monochrome subgraph is completely determined by
the underlying non-directed graph, and we have an edge-coloured complete
graph in the usual sense (see, for example, \cite{r8}, ch.~6).

The basic adjacency matrices $B_1,\ldots,B_{r-1}$ are the adjacency matrices
in the usual sense for the monochrome subgraphs of $C(G,\Omega)$. Products of
these matrices enumerate paths in $C(G,\Omega)$ in the following way:
\[B_{i_1}B_{i_2}\cdots B_{i_k}=(m_{\alpha\beta})\]
where $m_{\alpha\beta}$ is the number of paths of length $k$ from $\alpha$ to
$\beta$ whise first step is an edge of colour $c_{i_1}$, second step an
edge of colour $c_{i_2}$, \dots, last step an edge of colour $c_{i_k}$.

If we amalgamate two or more colours in $C(G,\Omega)$ the resulting coloured
graph will still admit $G$ as a group of automorphisms, though now no longer
flag-transitive, and the property that if $\Delta$ is a monochrome subgraph
then so is $\Delta^*$ may not persist. This operation of combining colours
corresponds to adding the appropriate basic adjacency matrices. Accordingly
we define a general \emph{adjacency matrix} for $C(G,\Omega)$ to be any
\label{p16}%
matrix of the form $\sum_{i\in I}B_i$ where $I$ is a non-empty subset of
$\{1,2,\ldots,r-1\}$, and we define the (generalized) subdegree of this
adjacency matrix to be its row -- or column -- sum. A special case will be
used in the proof of Lemma~\ref{l613}: if $\Gamma$, $\Delta$ are monochrome
subgraphs (possibly after some amalgamation of colours) we define a subgraph
$\Gamma\circ\Delta$ by
\[\begin{array}{rcl}
\Gamma\circ\Delta&=&\{(\alpha,\beta)\mid\alpha\ne\beta\hbox{ and there exists }
\gamma\in\Gamma\hbox{ such that }\\&&\qquad(\alpha,\gamma)\in\Gamma\hbox{ and }
(\gamma,\beta)\in\Delta\}.\end{array}\]
That is, $(\alpha,\beta)\in\Gamma\circ\Delta$ if and only if $\alpha\ne\beta$
and there is a path of length~$2$ directed from $\alpha$ to $\beta$, whose
first leg is coloured with the $\Gamma$-colour and whose second leg is the
colour corresponding to $\Delta$. It can happen that $\Gamma\circ\Delta$ is
empty, but only if $\Gamma=\Delta^*$ and $\Delta$ has subdegree~$1$. A
primitive group which is not regular (i.e.\ not cyclic of prime order) never
has non-trvial suborbits of subdegree~$1$, and so in such a case
$\Gamma\circ\Delta$ is never empty. It is of course a very special property 
of the coloured graph $C(G,\Omega)$ that $\Gamma\circ\Delta$ is a union of
monochrome subgraphs.

It has been remarked above that the basic adjacency matrices
$B_0,B_1,\ldots,B_{r-1}$ span an algebra $V$. Therefore there exist complex
numbers $a_{ijk}$ such that
\begin{equation}
B_iB_j=\sum_{i=0}^{r-1}a_{ijk}B_k.
\tag{6.1}
\end{equation}
In fact, these \emph{multiplication constants} $a_{ijk}$ are non-negative
integers, for, if the $(\alpha,\beta)$ entry of $B_k$ is $1$ then the
$(\alpha,\beta)$ entries of all the other basic adjacency matrices are~$0$,
therefore $a_{ijk}$ is the $(\alpha,\beta)$ entry in $B_iB_j$: and the
coefficients of $B_iB_j$ are obviously non-negative integers. The 
multiplication constants have an obvious geometric significance:
\[\begin{array}{rcl}
a_{0jk}&=&\delta_{jk}\\
a_{jok}&=&\delta_{ik}\\
a_{ij0}&=&\delta_{ij^*}n_i
\end{array}\]
(where $n_i$ is the appropriate subdegree), and, if $i\ge1$, $j\ge1$, $k\ge1$,
then $a_{ijk}$ is the number of oriented triangles $(\alpha,\beta,\gamma)$ on
a fixed base $(\alpha,\beta)$ of colour $c_k$, whose other two sides,
$(\alpha,\gamma)$ and $(\gamma,\beta)$ are coloured $c_i$ and $c_j$
respectively. Notice that $n_ka_{ijk}$ is the number of oriented triangles at
each vertex having edges of colours $c_i,c_j,c_{k^*}$ in that order.

\paragraph{N.B.} These structure constants, and numbers closely related to
them, have appeared many times before in the literature: see, for example,
Frame~\cite{r4,r5,r6,r7} and D.G.~Higman~\cite{r9,r10}. In this last paper
\cite{r10} Higman's $\mu_{ij}^{(\alpha)}$ is our $a_{i\alpha^*j}$, Consequently,
if we define (as Higman does) $M_\alpha=(\mu_{ij}^{(\alpha)})_{ij}$, then
(6.1) shows that $M_{\alpha^*}$ is the matrix representing $B_\alpha$ 
with respect to the basis $\{B_0,\ldots,B_{r-1}\}$ of $V$ in the regular
representation of $V$ (compare \cite{r10}, p.31). There are many relations
between the constants $a_{ijk}$ (compare \cite{r5} Theorem 2.9(b), (c), or
see \cite{r10}, (4.1), (4.2)) which are geometrically almost obvious, or which
can be read off from equation (6.4) below. The less obvious relations,
\[\sum_{\nu}a_{ij\nu}a_{\nu kl} = \sum_{\nu}a_{jk\nu}a_{i\nu l},\]
which express the associativity of $V$, also have geometric significance.
The left side enumerates quadrilaterals $(\alpha,\gamma,\delta,\beta)$ on a
fixed edge $(\alpha,\beta)$ of colour $c_i$, whose other sides
$(\alpha,\gamma)$, $(\gamma,\delta)$, $(\delta,\beta)$ are coloured $c_i$,
$c_j$, $c_k$ respectively, by counting those in which the diagonal
$(\alpha,\delta)$ is coloured $c_\nu$, and summing. The right side of the
equation enumerates the same quadrilaterals, but by counting those in which
the diagonal $(\gamma,\beta)$ has colour $c_\nu$ and summing.

The constants $a_{ijk}$ can be computed from trace relations. From the
definition of the matrices $B_i$ it is clear that
\begin{equation}
\tr(B_i)=
\begin{cases}
n & \hbox{if }i=0\\
0 & \hbox{if }i\ne0\\
\end{cases}.
\tag{6.2}
\end{equation}
From (6.2) and the defining relations (6.1) of $V$ it follows that
\begin{equation}
\tr(B_iB_j)=na_{ij0}
=\begin{cases}
0 & \hbox{if }j\ne i^*\\
nn_i & \hbox{if }j=i^*\\
\end{cases},
\tag{6.3}
\end{equation}
(see \cite{r20}, Theorem 28.10). Furthermore,
\[B_iB_jB_{k^*}=\sum_l a_{ijl}B_lB_{k^*},\]
and so (compare \cite{r6}, Equation (4.2)) from the above equation,
\begin{equation}
\tr(B_iB_jB_{k^*}) = nn_ka_{ijk}.
\tag{6.4}
\end{equation}
This again may be seen geometrically, for $B_iB_jB_{k^*}$ has as its
$(\alpha,\alpha)$ coordinate the number of paths of length $3$ from $\alpha$
back to $\alpha$ coloured $c_i$, $c_j$, $c_{k^*}$ in that order. This number
is $n_ka_{ijk}$ for all $\alpha\in\Omega$, and so the trace of
$B_iB_jB_{k^*}$ is $nn_ka_{ijk}$.

The equation
\begin{equation}
\sum_0^{r-1}B_i=W,
\tag{6.5}
\end{equation}
where $W$ is the matrix all of whose entries are $1$, expresses the fact that
the monochrome subgraphs $\Delta_i$ cover the whole complete graph on 
$\Omega\times\Omega$ without overlapping.

We shall find a slight generalisation of equations (6.2), (6.3) and (6.5)
useful. Namely, let $C$ be a complete directed graph obtained from
$C(G,\Omega)$ by amalgamating colours according to the rule%
\footnote{Such a collection of the resulting matrices will be called an
\emph{admissible set} of adjacency matrices.}
that, if $\Delta$ is a monochrome subgraph in $C$, then so also is $\Delta^*$.
The adjacency matrices $A_1,A_2,\ldots,A_{t-1}$ ($t\le r$) are some of the
general adjacency matrices of $C(G,\Omega)$ defined on page~\pageref{p16}.
If we put $A_0=I$, and extend the convention used before so that 
$A_{i^*}=A_i^\top$ then it is easy to see that the linear and quadratic trace
relations hold (compare \cite{r15}):
\begin{equation}
\tr(A_i)=\begin{cases}
n & \hbox{if }i=0\\
0 & \hbox{if }i\ge1
\end{cases},
\tag{$6.2^\#$}
\end{equation}
\begin{equation}
\tr(A_iA_j)=\begin{cases}
0 & \hbox{if }j\ne i^*\\
nm_i & \hbox{if }j=i^*
\end{cases},
\tag{$6.3^\#$}
\end{equation}\label{p18}%
where $m_i$ now is the row sum (or column sum) of $A_i$, that is, it is the
number of edges of $C$ of colour $c_i$ emanating from each point. Of course
we will also have
\begin{equation}
\sum_0^{t-1}A_i=W.
\tag{$6.5^\#$}
\end{equation}

Since $V$ and the algebra spanned by $G$ (as a set of permutation matrices)
are centralizer algebras of one another we know that they are simultaneously
reducible\footnote{We assume that the coefficient field is the splitting
field for $G$, or at least, that the matrix representations $D_\lambda$ are
absolutely irreducible} and there is an invertible matrix $U$ (see
\cite{r20}, \S29) such that
\[\scriptsize{U^{-1}gU=\begin{pmatrix}
1&&&&\\
&e_1\left\{\begin{matrix}D_1(g)&&\\&\ddots&\\&&D_1(g)\end{matrix}\right.&&&\\
&&e_2\left\{\begin{matrix}D_2(g)&&\\&\ddots&\\&&D_2(g)\end{matrix}\right.&&\\
&&&\ddots&\\
&&&&e_s\left\{\begin{matrix}D_s(g)&&\\&\ddots&\\&&D_s(g)\end{matrix}\right.
\end{pmatrix}}\]
for all $g\in G$; and, for any adjacency matrix $A_i$,
\[U^{-1}A_iU=\begin{pmatrix}
m_i&&&&\\
&\Theta_{i1}\times I_{f_1}&&&\\
&&\Theta_{i2}\times I_{f_2}&&\\
&&&\ddots&\\
&&&&\Theta_{is}\times I_{f_s}
\end{pmatrix}.\]\label{p19}
Here $D_\lambda$ is a matrix representation of $G$ of degree $f_\lambda$ 
which affords the character $\chi_\lambda$ (where
$\pi=1+\sum_1^se_\lambda\chi_\lambda$); $m_i$ is the subdegree of $A_i$;
$\Theta_{i\lambda}$ is an $e_\lambda$ by $e_\lambda$ matrix, $I_{f_\lambda}$
is the $f_\lambda$ by $f_\lambda$ identity matrix, and $\times$ denotes
Kronecker product. Since conjugation by $U$ gives an algebra isomorphism which
leaves traces invariant, equations ($6.2^\#$), ($6.3^\#$), (6.4) become
\begin{equation}
\sum_{\lambda=0}^sf_\lambda\tr(\Theta_{i\lambda})=0\quad(i\ge1)
\tag{6.6}
\end{equation}
\begin{equation}
\sum_{\lambda=0}^sf_\lambda\tr(\Theta_{i\lambda}\Theta_{j\lambda})=
\begin{cases}
0 & \hbox{if }j\ne i^*\\ nm_i & \hbox{if }j=i^*
\end{cases}
\tag{6.6}
\end{equation}
\begin{equation}
\sum_{\lambda=0}^sf_\lambda\tr(\Theta_{i\lambda}\Theta_{j\lambda}\Theta_{k\lambda})=nn_ka_{ijk},
\tag{6.8}
\end{equation}
where, in (6.6) and (6.7) the matrices $\Theta_{?\lambda}$ can be interpreted
as coming from an admissible set of general adjacency matrices
$A_1,\ldots,A_{t-1}$ as on page~\pageref{p18}, but in the third equation the
matrices $\Theta_{?\lambda}$ must come from the basic adjacency matrices
$B_1,\ldots,B_{r-1}$. Since 
\[U^{-1}WU=\begin{pmatrix} 
n&0&\ldots&0\\0&&&\\\vdots&&&\\0&&\ldots&0\end{pmatrix},\]
($6.5^\#$) gives
\begin{equation}
\sum_{i=1}^{t-1}\Theta_{i\lambda}=-I_{e_\lambda}\qquad(\lambda=1,\ldots,s)
\tag{6.9}
\end{equation}
where the $\Theta_{i\lambda}$ can be interpreted as coming from an 
admissible set of general adjacency matrices $A_1,\ldots,A_{t-1}$.

If $e_\lambda=1$ then $\Theta_\lambda$ is just an eigenvalue $\theta_\lambda$
of the adjacency matrix $A$. In any case, if $\theta$ is an eigenvalue of
$\Theta_\lambda$ then $\theta$ is an eigenvalue of $A$: we shall call the 
eigenvalues of $\Theta_\lambda$ the eigenvalues of $A$ \emph{associated with}
$\chi_\lambda$, and of course it can happen that one eigenvalue of $A$ is
associated with several different constituents of $\pi$. In case \emph{all}
the multiplicities $e_\lambda$ are $1$, so that $V$ is commutative (and
$s=r-1$, the relations (6.6)(6.9), (6.7), (6.8) are just linear, quadratic
and cubic relations on the eigenvalues of the adjacency matrices. This is
the case in which they are the most useful. Two small facts about this case
are useful:

\begin{unnumber}{\bf(6.10)}
If $r=3$ then $\theta_{i1}\ne\theta_{i2}$\quad$(i=1,2)$;
\end{unnumber}

\begin{unnumber}{\bf(6.11)}
If $V$ is commutative (i.e.\ $e_\lambda=1$ for all $\lambda$). if the degrees
of $f_\lambda$ are all different, and if the eigenvalues of the adjacency
matrix $A$ associated with distinct constituents $\chi_\lambda$ are different,
then $A$ is symmetric (i.e.\ the corresponding colour is self-paired).
\end{unnumber}

\begin{proof}
If $r=3$ then $V$ is automatically commutative. If $\theta_{11}=\theta_{12}$
then (6.9) gives
\[\theta_{21}=-1-\theta_{11}=-1-\theta_{12}=\theta_{22}.\]
But then the matrices $O$, $U^{-1}B_1U$, $U^{-1}B_2U$ would be linearly
dependent, which is not so. Thus $\theta_{11}\ne\theta_{12}$, and similarly
$\theta_{21}\ne\theta_{22}$.

If $V$ is commutative then $U^{-1}AU$ and $U^{-1}A^\top U$ are diagonal
matrices. Since $A$ and $A^\top$ have the same eigenvalues with the same 
multiplicities, the hypotheses of (6.11) ensure that $U^{-1}AU=U^{-1}A^\top U$.
That is, $A=A^\top$.
\end{proof}

\setcounter{theorem}{11}

\begin{lemma}
Let $\chi$, $\psi$ be irreducible constituents of $\pi$ whose
multiplicities are~$1$, and let $\theta$, $\phi$ be associated eigenvalues
of an adjacency matrix $A$. If $\gamma\in\Gamma$ (the Galois group of a
suitable field, cf.\ \S2) and $\chi^\gamma=\psi$, then $\theta\gamma=\phi$.
\label{l612}
\end{lemma}

\begin{proof}
Use a superscript $\gamma$ to indicate the natural operation of $\Gamma$ on
$F^n$, the vector space of $1\times n$ row vectors over $F$, and on
$M_n(F)$, the set of all $n\times n$ matrices over $F$. Since
$G\subseteq M_n(F)$, $G$ operates on $F^n$ by right multiplication and $\pi$
is the character of this representation of $G$. Let $E(\chi)$ be the 
$G$-invariant subspace of $F^n$ which affords $\phi$. If 
$\{\mathbf{x}_i\mid i=1,\ldots,f\}$ is a basis for $E(\chi)$, and if for
$g\in G$,
\[\mathbf{x}_i^\gamma g=\sum_j\lambda_{ij}(g)\mathrm{x}_j^\gamma,\]
then, since $g$ is a matrix with rational coefficients,
\[\mathbf{x}_i^\gamma g=\sum_j\lambda_{ij}(g)\gamma\mathbf{x}_j^\gamma.\]
We read off from this that $\{\mathbf{x}_i^\gamma\mid i=1,\ldots,f\}$
is a basis for the subspace of $F^n$ affording the character $\chi^\gamma$
of $G$. That is, $E(\chi)^\gamma=E(\psi)$. Now
\[\mathbf{x}A=\theta\mathbf{x}\]
for all $x\in E(\chi)$. Therefore
\[x^\gamma A =\theta\gamma\mathrm{x}^\gamma\]
for all $\mathbf{x}\in E(\chi)$, since $A$ is a rational matrix. Thus
$\theta\gamma$ is the eigenvalue of $A$ on $E(\chi)^\gamma$, that is, by
definition of $\phi$ we have $\theta\gamma=\phi$.
\end{proof}

\paragraph{N.B.} This lemma has an obvious generalisation in case the
multiplicities of $\chi$, $\psi$ in $\pi$ are not $1$, namely, the set of
eigenvalues of $A$ associated with $\chi$ is carried by $\gamma$ to the set
of eigenvalues associated with $\psi$. And it has an easy converse: is
$\theta$, $\phi$ are eigenvalues of $A$ such that $\theta\gamma=\phi$, then
$\chi^\gamma=\psi$ where
\begin{eqnarray*}
\chi&=&\{\sum e_i\chi_i\mid\theta\hbox{ is an eigenvalue associated with }
\chi_i\},\\
\psi&=&\{\sum e_i\chi_i\mid\phi\hbox{ is an eigenvalue associated with }
\chi_i\}.
\end{eqnarray*}
However, for both the generalisation and its converse we may need to work in a
larger field than $F$, namely a normal extension which contains all the
eigenvalues of the appropriate matrices $\Theta_\lambda$ as on page
\pageref{p19}.

Our last lemma in this section is the filter which sorts primitive from
imprimitive groups in the calculation of \S\S7--10. D.~G.~Higman~\cite{r10}
and C.~C.~Sims~\cite{r17} have pointed out that $G$ is primitive on $\Omega$
if and only if all the monochrome subgraphs in $C(G,\Omega)$ are connected.
This can be translated into a criterion applying to the eigenvalues of the
aduacency matrices, a slight modification of a special case of the
Perron--Frobenius Theorem (see, for example, \cite{r2}\footnote{I am grateful
to Dr J.~A.~D.~Welsh for introducing me to relevant literature.}). For
completeness a proof is included.

\begin{lemma}
Suppose that $G$ is primitive on $\Omega$. Let $A$ be an adjacency matrix of 
subdegree~$m$ where $m>1$. Then
\begin{itemize}
\item[(i)] $m$ is an eigenvalue of $A$ of (algebraic) multiplicity~$1$;
\item[(ii)] if $\theta$ is any other eigenvalue of $A$ then $|\theta|<m$.
\end{itemize}
\label{l613}
\end{lemma}

\begin{proof}
The matrix $A$ operates by right multiplication as a linear transformation
of $\mathbb{C}^n$, the space of $1$~by~$n$ row vectors
$\mathbf{x}=(x_0,\ldots,x_{n-1})$ with complex coordinates. If
$\mathbf{t}=(1,1,\ldots,1)$ then, since every column of $A$ contains $m$
entries $1$ and $n-m$ entries $0$, we have
\[\mathbf{t}A=m\mathbf{t}.\]
Hence $m$ is an eigenvalue of $A$ with (algebraic) multiplicitiy at least~$1$.
Let $U=\{\mathbf{x}\mid\sum x_\alpha=0\}$. Then
$\mathbb{C}^n=\langle t\rangle\oplus U$, and $U$ is invariant under $A$. Thus
it will be sufficient to show that if $\theta$ is an eigenvalue of $A$ on $U$
then $|\theta|<m$.

Now let $\mathbf{x}$ be any vector in $\mathbb{C}^n$, and let
$\mathbf{y}=\mathbf{x}A$. If the $(\alpha,\beta)$ coordinate of $A$ is
$m_{\alpha\beta}$ then
\[|\mathbf{y}|^2=\sum_\beta\Big|\sum_\alpha x_\alpha m_{\alpha\beta}\Big|^2.\]
Now $m_{\alpha\beta}$ is $0$ or $1$, so $m_{\alpha\beta}^2=m_{\alpha\beta}$, and
\begin{eqnarray*}
\Big|\sum_\alpha x_\alpha m_{\alpha\beta}\Big|^2
&=& \Big|\sum_\alpha x_\alpha m_{\alpha\beta}^2\Big|^2\\
&\le& \left(\big|\sum_\alpha x_\alpha m_{\alpha\beta}\big|^2\right)
\left(\sum_{\alpha}|m_{\alpha\beta}|^2\right)
\end{eqnarray*}
by Cauchy's Inequality. Each column of $A$ contains precisely $m$ entries
equal to $1$, the rest $0$, and so $\sum_\alpha|m_{\alpha\beta}|^2=m$.

Thus
\begin{eqnarray*}
\|\mathbf{y}\|^2 &\le& m\sum_\beta\sum_\alpha|x_\alpha m_{\alpha\beta}|^2\\
&=& \sum_\alpha\sum_\beta m_{\alpha\beta}^2|x_\alpha|^2\\
&=& m\sum_\alpha m|x_\alpha|^2\\
&=& m^2\|\mathbf{x}\|^2,
\end{eqnarray*}
and so $\|\mathbf{y}\|\le m\|\mathbf{x}\|$.

If equality is to hold then, returning to our application of Cauchy's
Inequality, we require that for every $\beta$ the coordinates $x_\alpha$ for
which $m_{\alpha\beta}=1$ are all equal. That is, $x_\alpha=x_\gamma$ if there
exists $\beta$ such that $(\alpha,\beta)\in\Delta$ and
$(\gamma,\beta)\in\Delta$, where $\Delta$ is the graph of which $A$ is
adjacency matrix. In other words, $x_\alpha=x_\gamma$ if
$(\alpha,\beta)\in\Delta\circ\Delta^*$. Now $\Delta\circ\Delta^*$ is not empty
since the subdegree of $\Delta$ is not~$1$ (cf.~p.~\pageref{p16}), and
therefore $\Delta\circ\Delta^*$ is a union of monochrome subgraphs of
$C(G,\Delta)$. These are all connected since $G$ is primitive, and consequently
$\Delta\circ\Delta^*$ is connected. It now follows that $x_\alpha=x_\gamma$
for all $\alpha,\gamma\in\Omega$, and $\mathbf{x}$ is a multiple of
$\mathbf{t}$. Thus if $\mathbf{x}\notin\langle\mathbf{t}\rangle$ then
$\|\mathbf{x}A\|<m\|\mathbf{x}\|$. In particular, if $\theta$ is an 
eigenvalue of $A$ on $U$ then $|\theta|<m$, and the lemma is proved.
\end{proof}

\paragraph{N.B.} The condition $m>1$ is necessary but not restrictive, for,
if $m=1$ then $G$ must be cyclic of prime order $p$ and $A$ must be one of
the basic adjacency matrices. In this case $A$ is a permutation of order~$p$
whose eigenvalues are $p$-th roots of~$1$. Primitivity of $G$ is also
necessary, for, if $G$ is imprimitive, having $k$ blocks of size $l$, then
by suitably re-numbering the elements of $\Omega$ we can ensure that $A$ is
a matrix in block form $(B_{ij})$, where $N_{ij}$ is an $l\times l$ zero 
matrix if $i\ne j$, and $B_{ii}$ is an $l\times l$ matrix of zeros and ones
having $m$ entries equal to $1$ in each row and column. It is easy to see that
in this case $A$ has at least $k$ linearly independent eigenvectors with
eigenvalue~$m$.

The linear and quadratic equations in (6.6), (6.7), (6.9) were used by
Frame and D.~G.~Higman \cite{r4,r5,r6,r7,r9,r10} to obtain information about
the degrees $f_\lambda$ of the irreducible constituents of $\pi$ from 
information about the subdegrees $n_i$. They have been used by Wielandt and
Ito~\cite{r19,r13} the other way round, to deduce the subdegrees $n_i$ from
knowledge of the numbers~$f_\lambda$. This is the program to be carried out
now. In most cases $V$ turns out to be commutative, and we take the linear and
quadratic equations (6.6), (6.7), (6.9) as equations for the eigenvalues of
the adjacency matrices. As such we have about $r^2$ unknowns and only about
$\frac{1}{2}r(r+3)$ equations, where $r$ is the rank -- and these equations
seem not to be independent in general. On the other hand, since the adjacency
matrices have integer coordinates, their eigenvalues are algebraic integers,
and (6.6)--(6.9) is a system of diophantine equations. The cubic equations
(6.8) introduce new unknowns $a_{ijk}$, one for each equation. Nevertheless
they are useful as divisibility conditions. There are of course higher trace
relations, but these give no new information, because once equations
(6.6)--(6.8) have been solved to give the eigenvalues, they give the structure
constants $a_{ijk}$ and hence full information about the algebra $V$. The
policy therefore will be to find all solutions of (6.6)--(6.9) for which the
numbers $a_{ijk}$ are non-negative rational integers.

In \S\S7--10 there are never more than $4$ eigenvalues of the matrices
$B_i$ (or $A_i$).  Of these one is $n_i$ (or $m_i$). The others will be
re-named $\lambda_i$, $\mu_i$, $\nu_i$ for simplicity.

\section{Case I does not arise}

In this case $r$ is $3$ and (6.6) gives
\[n_i+\left(\frac{3p-1}{2}\right)(\lambda_i+\mu_i)=0.\]
Now $\lambda_i$. $\mu_i$ are algebraic integers, and so $\lambda_i+\mu_i$,
being rational, is a rational integer. Hence $n_1,n_2$ are both multiples of
$(3p-1)/2$. But $n_1+n_2=3p-1$, and therefore $n_1=n_2=(3p-1)/2$
(compare~\cite{r20}, Theorem~30.2).

Suppose now that $\Delta_1$ is self-paired. Then, from (6.6), (6.7),
\begin{eqnarray*}
\lambda_1+\mu_1 &=& -1 \\
\lambda_1^2+\mu_1^2 &=& \frac{3p+1}{2}.
\end{eqnarray*}
This yields $\lambda_1\mu_2=-(3p-1)/4$, and since $\lambda_1,\mu_1$ are
algebraic integers we must have $3p\equiv1\pmod{4}$, that is,
\[p\equiv3\pmod{4}.\]
It follows (Lemma~\ref{l51}) that $\overline{\chi_1}=\chi_2$, and so, from
Lemma~\ref{l612}, $\lambda_1$, $\mu_1$ are complex conjugate. But this is
certainly not the case -- the equations above give
$\lambda_1,\mu_1=(-1\pm\sqrt{3p})/2$.

Suppose therefore that $\Delta_1$ is paired with $\Delta_2$. Then
\begin{eqnarray*}
\lambda_1+\mu_1 &=& -1 \\
\lambda_1^2+\mu_1^2 &=& -\frac{3p-1}{2}.
\end{eqnarray*}
In this case therefore $\lambda_1\mu_1=(3p+1)/4$, so that $p\equiv1\pmod{4}$.
Applying Lemmas~\ref{l51} and~\ref{l612} we see that $\lambda_1,\mu_1$  must
both be real whereas the equations clearly give complex conjugate values.
These two contradictions prove that Case~I cannot occur.

\section{Groups of type II}

\begin{theorem}
If a group of type II exists, then either
\begin{itemize}
\item[(i)] $p=48a^2+30a+5$\hfil\break
and the subdegrees are $1$, $p+4a+1$, $2p-4a-2$; or
\item[(ii)] $p=48a^2+66a+23$\hfil\break
and the subdegrees are $1$, $p-4a-3$, $2p+4a+2$.
\end{itemize}
\label{t61}
\end{theorem}
Here $a\ge0$ and $a$ is an integer. The groups $A_6$, $S_6$ as in
Example~1 (p.~\pageref{p2}) are instances of case (i) above with $a=0$.
In \S12 we will show that in case (i) $a$ is even, and in case (ii) $a$
must be odd.

\begin{proof}
Groups of type II are of rank~$3$ and both their non-trivial suborbits are
self-paired ((6.10) and (6.11)). We have (6.6)
\[n_i+p\lambda_i+(2p-1)\mu_i=0\]
and $\lambda_i$, $\mu_i$ are rational integers. Consequently
\[n_i\equiv \mu_i\pmod{p}\]
and so $n_i=\epsilon_i+\mu_i$. Since $n_i>0$, from Lemma~\ref{l613} we
know that $\epsilon_i>0$, and from (6.9), $\epsilon_1+\epsilon_2=3$. Choosing
the notation suitably we may therefore suppose that $\epsilon_1=1$,
$\epsilon_2=2$. Thus $n_1=p+\mu_1$ and $\lambda_1=-2\mu_1-1$. Now (6.7) gives
\[(p+\mu_1)^2+p(1+2\mu_1)^2+(2p-1)\mu_1^2=3p(p+\mu_1),\]
which simplifies:
\begin{eqnarray*}
&&6\mu_1^2+3\mu_1+1-2p=0.\\
&&\mu_1=\frac{1}{4}\left(-1\pm\sqrt{\frac{16p-5}{3}}\right).
\end{eqnarray*}
Since $\mu_1$ is a rational integer it follows that $16p-5=3b^2$ for some
positive integer $b$. We require $p\equiv3$, $5$, $11$ or $13$ (mod~$16$)
if $3b^2+5\equiv0$, and $b\equiv5$ or $11$ if $3b^2+5$ is not to be divisible
by $32$. Hence $b=16a+5$ or $b=16a+11$, $a\ge0$. In the former case
$\mu_1=(-1+b)/4$ since $-1-b$ is not divisible by $4$, and in the latter
case $\mu_1=(-1-b)/4$. Now computing $p$ and $n_1$ in terms of $a$ gives the
statement of the theorem.
\end{proof}

\section{Cases III, IV, V}

\begin{theorem}
If $G$ is of Type III then $p=3a^2+3a+1$ for some integer $a$, and the
subdegrees are either
\begin{itemize}
\item[(i)] $1,p-2a-1,2p+2a$; or
\item[(ii)] $1,p+2a+1,2p-2a-2$.
\end{itemize}
\label{t91}
\end{theorem}

The groups $A_7$, $S_7$ of degree~$21$ as in Example~2 (p.~\pageref{p2}) are
instances of Type~III, case (ii) with $a=1$.

\begin{theorem}
If $G$ is of Type IV, then $p=3a^2+3a+1$ where $a$ is an integer. In case two
suborbits are paired, say $\Delta_2^*=\Delta_3$, then
\begin{itemize}
\item[(i)] $a$ is even and $n_1=p-2a-1,n_2=n_3=p+a$; or
\item[(ii)] $a$ is odd and $n_1=p+2a+1,n_2=n_3=p-a-1$.
\end{itemize}
Otherwise all suborbits are self-paired, in which case
\begin{itemize}
\item[(iii)] $a$ is even and the subdegrees are $1, p+2a+1, p-a-1,p-a-1$.
\end{itemize}
\label{t92}
\end{theorem}

\begin{theorem}
Case V does not arise.
\label{t93}
\end{theorem}

\paragraph{Proof of 9.1} Again our group has rank~$3$ and both its non-trivial
suborbits are self-paired. Using (6.6),
\[n_i+2p\lambda_i+(p-1)\mu_i=0\]
so that $n_i=\epsilon_ip+\mu_i$. As before, (6.13) gives $\epsilon_i\ge1$,
and from (6.9), $\epsilon_1+\epsilon_2=3$. We choose the notation so that
$\epsilon_1=1$, $n_1=p+\mu_1$, $\lambda_1=(-\mu_1-1)/2$. Now from (6.7),
\[(p+\mu_1)^2+2p\left(\frac{\mu_1+1}{2}\right)^2+(p-1)\mu_1^2=3p(p+\mu_1),\]
\[3\mu_1^2=4p-1.\]
Since again $\mu_1$ is a rational integer, we may put $(4p-1)/3=(2a+1)^2$
where $a$ is an integer, $a\ge0$, and then $p=3a^2+3a+1$, $\mu_1=\pm(2a+1)$.
Hence $n_1=p\pm(2a+1)$, and Theorem \ref{t91} follows.

\paragraph{Proof of 9.3} A group of type V would have rank~$6$, and there are
$5$ basic adjacency matrices $B_1,\ldots,B_5$ to be considered. The matrices
$\Theta_{i,1}$ are size $2\times 2$, and their eigenvalues $\lambda_i,\mu_i$
are eigenvalues of $B_i$ of multiplicity~$p$. The fourth eigenvalue $\nu_i$
of $B_i$ is necessarily a rational integer. The equation (6.6) gives
\[n_i+p(\lambda_i+\mu_i)+(p-1)\nu_i=0.\]
Now $\lambda_i+\mu_i$ is an algebraic integer which is rational, hence a
rational integer, and we get $n_i=\epsilon_ip+\nu+i$. As before, by
Lemma~\ref{l613} we must have $\epsilon_i\ge1$ for all $i$< and so 
$\sum\epsilon_i\ge5$. On the other hand, by (6.9) $\sum\nu_i=-1$, and
\[3p-1=\sum_1^5\nu_i=(\sum\epsilon_i)p+\sum\nu_i=(\sum\epsilon_i)p-1,\]
so that $\sum\epsilon_i=3$. This contradiction proves that case V cannot
arise.

\paragraph{Proof of 9.2}
Groups of type IV have rank $4$ and the three basic adjacency matrices $B_i$
have eigenvalues $n_i,\lambda_i,\mu_i,\nu_i$ where $\nu_i$ is a rational
integer, and either $\lambda_i$, $\mu_i$ are rational integers or they are 
algebraically conjugate algebraic integers. Using the equation (6.6) we
get $n_i=\epsilon_ip+\nu_i$ where, from Lemma~\ref{l613} $\epsilon_i>1$, and
from (6.9), $\sum\epsilon_i=3$. Thus $\epsilon_1=\epsilon_2=\epsilon_3=1$ and
\begin{equation}
n_i=p+\nu_i\qquad i=1,2,3.
\tag{9.4}
\end{equation}

Suppose now that $\Delta_3=\Delta_2^*$. Then $B_3=B_2^*$ and the eigenvalues
of $B_2,B_3$ must be the same. Thus $\nu_3=\nu_2$ and either
$\lambda_3=\lambda_2$, $\mu_3=\mu_2$ or $\lambda_3=\mu_2$, $\mu_3=\lambda_2$.
However, if $\lambda_3=\lambda_2$ and $\mu_3=\mu_2$ then
$U^{-1}B_3U=U^{-1}B_2U$ (where $U$ is as on p.\pageref{p19}) so that
$B_3=B_2$ which is not so. Therefore $\lambda_3=\mu_2$ and $\mu_3=\lambda_2$.
Now amalgamating colours $c_3$ and $c_2$ gives admissible adjacency matrices
$A_1=B_1$ and $A_2=B_2+B_3$, and produces the situation which we have
analysed in the proof of 9.1 above. The result is that $p=3a^2+3a+1$ and
either
\begin{itemize}
\item[(i)] $n_1=p-2a-1$, $n_2=n_3=p+a$, or
\item[(ii)] $n_1=p+2a+1$, $n_2=n_3=p-a-1$.
\end{itemize}
We wish to show now that in case (i) $a$ must be even, and in case (ii) $a$
is odd.

In case (i) the eigenvalues of the matrices $B_i$ are as shown:
\[\begin{array}{c|cccc}
&n_i&\lambda_i&\mu_i&\nu_i\\\hline
B_1&p-2a-1 & a & a & -2a-1 \\
B_2 & p+a & \lambda & \mu & a \\
B_3 & p+a & \mu & \lambda & a\\
\end{array}\]
where $\lambda+\mu=-a-1$ and, from (6.7)
\[\lambda\mu=\textstyle{\frac{1}{2}}(2p-a-a^2)
=\textstyle{\frac{1}{2}}(5a+2)(a+1).\]
One of the equations (6.8) is
\[3p(p+a)a_{223}=(p+a)^3+p\lambda\mu(\lambda+\mu)+(p-1)a^3.\]
A straightforward calculation gives
\[a_{223}=a^2+(3a)/2,\]
from which we deduce, since $a_{223}$ is an integer, that $a$ is even.

Case (ii) is similar: in this case the eigenvalues of the matrices $B_i$ are
\[\begin{array}{c|cccc}
&n_i&\lambda_i&\mu_i&\nu_i\\\hline
B_1&p+2a+1 & -a-1 & -a-1 & 2a+1 \\
B_2 & p-a-1 & \lambda & \mu & -a-1 \\
B_3 & p-a-1 & \mu & \lambda & -a-1 \\
\end{array},\]
and (6.6),(6.7) give $\lambda+\mu=1$, $\lambda\mu=\frac{1}{2}a(5a+3)$.
Then (6.8):
\[3p(p-a-1)a_{223}=(p-a-1)^3+p\lambda\mu(\lambda+\mu)-(p-1)(a+1)^3.\]
This yields $a_{223}=a^2+(a-1)/2$ and therefore $a$ must be odd. This completes
our treatment of Theorem~\ref{t92} for the case when two suborbits are paired:
in fact no further restrictions on $a$ can be deduced by the methods of \S6.

For the remainder of this section we suppose that all three non-trivial
suborbits are self-paired. Using equations (9.4) we eliminate $n_i$ from
equations (6.6),(6.7):
\begin{eqnarray*}
\lambda_i+\mu_i&=&-1-\nu_i\\
\lambda_i^2+\mu_i^2&=&-\nu_i^2+\nu_i+2p.
\end{eqnarray*}
This gives $2\lambda_i\mu_i=1+3\nu_i+2\nu_i^2-2p$, from which we deduce that
\begin{equation}
\nu_i\hbox{ is odd}
\tag{9.5}
\end{equation}
and $\lambda_i,\nu_i$ are respectively
$\frac{1}{2}\left(-1-\nu_i\pm\sqrt{4p-1-3\nu_i^2}\right)$. Adjusting the sign
of the square root where necessary, we may take it that
\begin{eqnarray*}
\lambda_i &=& \textstyle{\frac{1}{2}}(-1-\nu_i+\sqrt{4p-1-3\nu_i^2})\\
\mu_i &=& \textstyle{\frac{1}{2}}(-1-\nu_i-\sqrt{4p-1-3\nu_i^2}).
\end{eqnarray*}
The adjacency matrices $A_i$ being symmetric, we know that their eigenvalues
are real. Hence
\begin{equation}
3\nu_i^2\le4p-1.
\tag{9.6}
\end{equation}
And, from (6.9) we have
\begin{equation}
\left\{\begin{array}{rcl}
\nu_1+\nu_2+\nu_3&=&-1\\[1mm]
\sqrt{4p-1-3\nu_1^2}+\sqrt{4p-1-3\nu_2^2}+\sqrt{4p-1-3}\nu_3^2&=&0.\\
\end{array}\right.
\tag{9.7}
\end{equation}
Eliminating $\nu_3$ and rationalising gives
\begin{eqnarray*}
\nu_1^2(3\nu_2+2p+1)+\nu_1(3\nu_2^2+2p\nu_2+4\nu_2+2p+1)&&\\
+(2p+1)(\nu_2^2+\nu_2)-2p(p-1)&=&0.
\end{eqnarray*}
Now 
\[3\nu_2^2+2p\nu_2+4\nu_2+2p+1=(3\nu_2+2p+1)(\nu_2+1),\]
and therefore $3\nu_2+2p+1$ must divide $(2p+1)(\nu_2^2+\nu_2)-2p(p-1)$.
Calculating modulo $e\nu_2+2p+1$ we eliminate $p$ from the equation
\[2(2p+1)(\nu_2^2+\nu_2)-4p(p-1)=0\]
to get $3(\nu_2+1)^2(2\nu_2+1)\equiv0\pmod{3\nu_2+2p+1}$.
There is complete symmetry between $\nu_1,\nu_2,\nu_3$, and so
\begin{equation}
(2p+3\nu_i+1)\hbox{ divides } 3(\nu_i+1)^2(2\nu_i+1).
\tag{9.8}
\end{equation}

So far we have worked only with the equations (6.6),(6.7),(6.9), but at this
point we need also the cubic equations (6.8). They give, among other relations,
that $3pn_i$ divides $n_i^3+p(\lambda_i^3+\mu_i^3)+(p-1)\nu_i^3$. Substitute
for $n_i$, $\lambda_i$, $\mu_i$ in terms of $\nu_i$, divide by $p$, multiply
by $2$, and rearrange terms:
\[6(p+\nu_i)\hbox{ divides }2p^2-6p-6\nu_i^2+2\nu_i^3+(1+\nu_i)(4\nu_i^2-\nu_i+1).\]
Now calculation modulo $2(p+\nu_i)$ yields $2p^2-6p\equiv 2\nu_i(\nu_i+3)$
and after a little simplification we find that
\begin{equation}
2(p+\nu_i)\hbox{ divides }(\nu_i+1)(2\nu_i+1)(3\nu_i+1).
\tag{9.9}
\end{equation}

Since (9.7) $\nu_1+\nu_2+\nu_3=-1$ not all of $\nu_1$, $\nu_2$, $\nu_3$ can
be negative. Let us choose $b$ to be one of $\nu_1,\nu_2,\nu_3$ such that
$b\ge0$. Put
\begin{eqnarray*}
(b+1)(2b+1)(3b+1) &=& u.2(p+b)\\
(b+1)(2b+1)(3b+3) &=& v.(2p+3b+1)
\end{eqnarray*}
so that, by (9.8) and (9.9), $u$ and $v$ are integers. Subtraction gives
\[2(b+1)(2b+1) = w.2(p+b)+v(b+1)\]
where $w=v-u$. We wish to show that $w=0$. To do this, observe first that
\begin{equation}
\begin{array}{rcl}
w=v-u &=& (b+1)(2b+1)\left(\displaystyle{\frac{3b+3}{2p+3b+1}-\frac{3b+1}{2(p+b)}
}\right)\\[1mm]
&=& \displaystyle{\frac{(b+1)(2b+1)(4p-1-3b^2)}{2(2p+3b+1)(p+b)}},
\end{array}
\tag{9.10}
\end{equation}
so that $w\ge0$ by (9.6).

To obtain an upper bound for $w$ we first eliminate $v$,
\[3(b+1)^3(2b+1)=(2p+3b+1)(2(b+1)(2b+1)-2w(p+b)\]
and then put $p+b=x$:
\[4wx^2-2(b+1)(4b+2-w)x+(b+1)^2(2b+1)(3b+1)=0.\]
Since $x$ is certainly real the discriminant of this quadratic function cannot
be negative. Thus
\[4(b+1)^2(4b+2-w)^2-16w(b+1)^2(2b+1)(3b+1)\ge0,\]
and so
\begin{equation}
\begin{array}{rcl}
(4b+2-w)^2 &\ge& 4w(2b+1)(3b+1)\\
&=& w(4b+2)(6b+2).
\end{array}
\tag{9.11}
\end{equation}
However, from equation (9.10) we have that
\[w=\frac{(b+1)(2b+1)(4p-1-3b^2)}{2(2p+3b+1)(p+b)}<2b+1.\]
And now, since $w\ge0$ it follows that
\[2b+1<4b+2-w\le 4b+2.\]
Therefore (9.11) clearly implies that $w\le0$.

We have proved now that $w=0$. Consequently, by (9.10), $4p-1=3b^2$. Since
$b$ must therefore be odd we put $b=2a+1$ where $a\ge0$, and we have
$p=3a^2+3a+1$. Furthermore, if we suppose it was $\nu_1$ that was $b$ then
from (9.7),
\begin{eqnarray*}
\nu_2^2 &=& \nu_3^2\\
\nu_2 &=& \pm\nu_3.
\end{eqnarray*}
But $\nu_2+\nu_3=-1-\nu_1\ne0$ and it follows that
\[\nu_2=\nu_3=(-1-\nu_1)/2.\]
Thus $\nu_1=2a+1$, $\nu_2=\nu_3=-a-1$. Finally since by (9.5) $\nu_2$ and
$\nu_3$ must be odd, there is the restriction that $a$ is even. And
\begin{eqnarray*}
n_1 &=& p+2a+1\\
n_2=n_3 &=& p-a-1
\end{eqnarray*}
as was claimed.

\section{Cases VI, VII and VIII}

\begin{theorem}
Case VI cannot arise.
\label{t101}
\end{theorem}

\begin{theorem}
If $G$ is of type VII then $p$ is $7$, $19$ or $31$.
\begin{itemize}
\item[(i)] If $p=7$, the subdegrees are $1$, $4$, $8$, $8$;
\item[(ii)] if $p=19$, the subdegrees are $1$, $6$, $20$, $30$;
\item[(iii)] if $p=31$, the subdegrees are $1$, $20$, $32$, $40$.
\end{itemize}
\label{t102}
\end{theorem}

Examples 3 and 4 (page~\pageref{p2}) are instances with $p=7$, $19$
respectively. I do not know if a primitive group of Type VII and degree~$93$
exists.

\begin{theorem}
Case VIII cannot arise.
\label{t103}
\end{theorem}

We begin the proofs with a lemma:

\begin{lemma}
No symmetric matrix for $G$ (other than $W-I$) has just three eigenvalues
with multiplicities $1$, $p+1$, $2(p-1)$.
\label{l104}
\end{lemma}

\begin{proof}
Suppose on the contrary that $A_1$ is a symmetric adjacency matrix whose
eigenvalues $m_1$ (the subdegree), $\lambda_1$, $\mu_1$ have multiplicities
$1$, $p+1$, $2(p-1)$ respectively. Let $A_2=W-I-A_1$. Then $A_2$ is also a
symmetric adjacency matrix whose eigenvalues $m_2$, $\lambda_2$, $\mu_2$ have
multiplicities $1$, $p+1$, $2(p-1)$, and $\{A_1,A_2\}$ is an admissible set.

Furthermore, $A_1$ and $A_2$ commute and therefore there is a matrix $U$ such
that $U^{-1}A_1U$, $U^{-1}A_2U$ are both diagonal. Now if $\lambda_1=\mu_1$
then $\lambda_2=\mu_2$ and $I$, $A_1$, $A_2$ would not be linearly independent.
But in fact $I$, $A_1$, $A_2$ are linearly independent (we are assuming that
$A_1\ne W-I$, $A_1\ne I$) and it must follow that $\lambda_1\ne\mu_1$,
$\lambda_2\ne\mu_2$/ Notice that $\lambda_i$, $\mu_i$ are rational integers.

Equations (6.6),(6.7) give
\begin{equation}
\left.
\begin{array}{rcl}
m_i+(p+1)\lambda_i+2(p-1)\mu_i&=&0\\[1mm]
m_i^2+(p+1)\lambda_i^2+2(p-1)\mu_i^2&=&3pm_i
\end{array}
\right\}.
\tag{10.4.1}
\end{equation}
Calculating modulo $p$:
\[
\left.
\begin{array}{rcl}
m_i&\equiv&2\mu_i-\lambda_i\\
m_i^2&\equiv&2\mu_i^2-\lambda_i^2
\end{array}
\right\}
\pmod{p}.
\]
Thus
\[2\mu_i-\lambda_i^2\equiv 2\mu_i^2-\lambda_i^2\pmod{p}\]
and this gives
\[
\left.
\begin{array}{rccl}
\hbox{and this gives }&mu_i&\equiv&\lambda_i\\
\hbox{hence also }&m_i&\equiv&\lambda_i
\end{array}
\right\}
\pmod{p}.
\]
Put $m_i=\epsilon_ip+\lambda_i$, $\mu_i=\eta_ip+\lambda_i$, substitute
in (10.4.1) and compute modulo $p^2$:
\[
\left.
\begin{array}{rcl}
\epsilon_ip+\lambda_i+(p+1)\lambda_i+2(p-1)\lambda_i-2\eta_ip&\equiv&0\\
2\epsilon_ip\lambda_i+\lambda_i^2+(p+1)\lambda_i^2+2(p-1)\lambda_i^2
-4\lambda_i\eta_ip&\equiv&3p\lambda_i
\end{array}
\right\}
\pmod{p^2}.
\]
Collect terms and divide by $p$:
\[
\left.
\begin{array}{rcl}
\epsilon_i+3\lambda_i-2\eta_i&\equiv&0\\
3\lambda_i^2+\lambda_i(2\epsilon_i-4\eta_i-3)&\equiv&0
\end{array}
\right\}
\pmod{p}.
\]
Since $1+\lambda_1+\lambda_2=0$ not both $\lambda_1$ and $\lambda_2$ are
divisible by $p$. and interchanging $A_1$ and $A_2$ if necessary, we may
therefore assume that $\lambda_1\not\equiv0\pmod{p}$. Then
\[
\left.
\begin{array}{rccl}
&3\lambda_1&\equiv&2\eta_1-\epsilon_1\\
\hbox{and }&3\lambda_1&\equiv&4\eta_1-2\epsilon_1+3
\end{array}
\right\}
\pmod{p}.
\]
Eliminating $2\eta_1-\epsilon_1$ gives $\lambda_1\equiv-1\pmod{p}$. It
follows that $m_1=p-1$ or $m_1=2p-1$. If $m_1=p-1$, then, since
$\lambda_1\equiv\mu_1\equiv-1\pmod{p}$ and by Lemma~\ref{l613}
$|\lambda_1|<m_1$, $|\mu_1|<m_1$, the only possibility is
$\lambda_1=\mu_1=-1$, which contradicts our earlier observation that
$\lambda_1\ne\mu_1$. But, if $m_1=2p-1$ then $m_2=p$, and again there is
only one possibility, namely $\lambda_2=\mu_2=0$, consistent with
Lemma~\ref{l613}. This also gives a contradiction, and the lemma is proved.
\end{proof}

\paragraph{Proof of 10.1} In case VI the basic adjacency matrices would have
three eigenvalues with multiplicities $1$, $p+1$, $2(p-1)$, and they would be
symmetric by (6.10) and (6.11). Thus the lemma shows that case VI cannot arise.

\begin{lemma}
If VII is the case then all three non-trivial suborbits are self-paired.
\label{l105}
\end{lemma}

\begin{proof}
If not then the three basic adjacency matrices would be $B_1,B_2,B_3$ where
$B_3=B_2^\top$ say. Then the characteristic equation of $B_3$ would be the
same as that of $B_2$ and so the eigenvalues of $B_3$ would be the same as
those of $B_2$ with the same multiplicities. Since
$U^{-1}B_3U\ne U^{-1}B_2U$ (where $U$ is on p.\pageref{p19}) it would follow
that $\mu_3=\nu_2$ and $\mu_2=\nu_3$. But then, since
\[\begin{array}{rrcl}
&\mu_1+\mu_2+\mu_3&=&-1\\
\hbox{and }&\nu_1+\nu_2+\nu_3&=&-1
\end{array}\]
we would have $\mu_1=\nu_1$. Thus $B_1$ would be a matrix of the kind which
Lemma~\ref{l104} forbids. Hence all three non-trivial suborbits are
self-paired.
\end{proof}

From here on we treat cases VII and VIII together. In both cases the linear
trace relation for an adjacency matrix $A$ becomes
\[m+(p+1)\lambda+(p-1)(\mu+\nu)=0.\]
Now $\lambda$ is a rational integer and $\mu,\nu$ are algebraic integers. We
know that for a suitable invertible matrix $U$,
\[U^{-1}B_iU=\begin{pmatrix}
n_i&&\\
&p+1\left\{\begin{matrix}\lambda_i&&\\&\ddots&\\&&\lambda_i\end{matrix}\right.
\\
&&\Theta_i\end{pmatrix}\]
amd therefore if we form a new adjacency matrix by amalgamating colours then we
must add the appropriate eigenvalues $\lambda_i$ to get the corresponding
eigenvalue $\lambda$ of $A$. In particular, if $\Delta_i$ is not self-paired
then we amalgamate the colours $c_i,c_{i^*}$ to obtain a symmetric matrix
whose eigenvalues include $2\lambda_i$ with multiplicity $p+1$, and 
$m_i=2n_i=\epsilon_i(p-1)-2\lambda_i$ where $\epsilon_i=2\eta_i$.

For the rest of this section we use the following notation:
$A_1,\ldots,A_{t-1}$ are the symmetric adjacency matrices obtained by 
amalgamating paired colours; the eigenvalues of $A_i$ are $m_i,\lambda_i,
\mu_i,\nu_i$ with multiplicities $1,p+1,p-1,p-1$ respectively; and $\epsilon_i$
is defined by the equation
\[m_i=\epsilon_i(p-1)-2\lambda_i.\]
We recall that $\lambda_i$ is a rational integer, $\mu_i,\nu_i$ are algebraic
integers and since $A_i$ is symmetric they are real. Moreover, if $A_i$ is not
a basic adjacency matrix, that is, if $A_i$ has arisen from an amalgamation of
paired colours (by \ref{l105} this can happen only in case VIII), then
$\epsilon_i,m_i,\lambda_i$ are all even. In case VII there are $3$ such
matrices and they are all basic (by Lemma~\ref{l105}); in case VIII there may
be $3$, $4$ or $5$ matrices. In any case, $\{A_1,\ldots,A_{t-1}\}$ is an
admissible set of adjacency matrices.

\begin{lemma}
$\sum\epsilon_i=3$ and $\epsilon_i\ge0$ for all $i$.
\label{l106}
\end{lemma}

\begin{proof} 
\begin{eqnarray*}
3p-1&=&\sum m_i\\
&=&(p-1)\sum\epsilon_i-2\sum\lambda_i
\end{eqnarray*}
amd $\sum\lambda_i=-1$ by (6.9). Hence $(p-1)(\sum\epsilon_i)=3p-3$ and so
$\sum\epsilon_i=3$.

To prove the second assertion suppose on the contrary that $\epsilon_i<0$.
Then, since $m_i\ge0$ we must have $\lambda_i<0$ and (Lemma~\ref{l613})
\[-\lambda_i<\epsilon_i(p-1)-2\lambda_i\]
so that $\lambda_i<\epsilon_i(p-1)$.
In particular, in this case $|\lambda_i|>p-1$ and so $|\lambda_i|\ge p$.

Since $\mu_i,\nu_i$ are real, $\mu_i^2\ge0$ and $\mu_i^2\ge0$. Therefore from
the quadratic trace relation (6.7) we get the inequality
\[m_i^2+(p+1)\lambda_i^2\le3pm_i,\]
that is,
\[(p+1)\lambda_i^2\le m_i(3p-m_i)\le (3p/2)^2.\]
Hence $\lambda_i^2<9p/4$ and so $|\lambda_i|<3\sqrt{p}/2<p$.
This contradicts our earlier inequality and proves the lemma.
\end{proof}

Substituting $-2\lambda_i$ for $m_i$ in the equation (6.7) and calculating
modulo $p-1$ easily yields that $p-1$ divides $6\lambda_i(\lambda_i+1)$. We
will need a refinement:

\begin{lemma}
$p-1$ divides $3\lambda_i(\lambda_i+1)$.
\label{l107}
\end{lemma}

\begin{proof}
In this case equations (6.6), (6.7) give
\begin{eqnarray*}
\mu_i+\nu_i &=& -\epsilon_i-\lambda_i\\
\hbox{and }(p-1)(\mu_i^2+\nu_i^2) &=& 3p(\epsilon_i(p-1)-2\lambda_i)
-(\epsilon_i(p-1)-2\lambda_i)^2-(p+1)\lambda_i^2.
\end{eqnarray*}
Now $\mu_i\nu_i$ is a rational integer, and
$2\mu_i\nu_i=(\mu_i+\nu_i)^2-(\lambda_i^2+\mu_i^2)$.
Calculating modulo $2(p-1)$ therefore gives
\[\left.\begin{array}{rcl}
0 &\equiv& (p-1)(\epsilon_i+\lambda_i)^2-3p(\epsilon_i(p-1)-2\lambda_i)^2
+(p+1)\lambda_i^2\\[1mm]
&\equiv& (p-1)(\epsilon_i+\lambda_i)^2-\epsilon_i(p-1)+6\lambda_i+4\lambda_i^2
+(p-1)\lambda_i^2+2\lambda_i^2\\[1mm]
&\equiv& (p-1)(\epsilon_i^2-\epsilon_i)+6\lambda_i+6\lambda_i^2
\end{array}\right\}\pmod{2(p-1)}.\]
Since $\epsilon_i^2-\epsilon_i$ is even we see that $2(p-1)$ divides
$6\lambda_i(\lambda_i+1)$ and therefore $p-1$ divides $3\lambda_i(\lambda_i+1)$
as claimed.
\end{proof}

The last paragraph of the proof of 10.6 was to show that
$|\lambda_i|<3\sqrt{p}/2$. Again we need a finer estimate:

\begin{lemma}
\[\epsilon_i(p-1)(6p-2\epsilon_ip+\epsilon_i)
-6\lambda_i(2p-\epsilon_ip+\epsilon_i)-(3p+9)\lambda_i^2\ge0.\]
\label{l108}
\end{lemma}

\begin{proof}
We look more closely at the discriminant of the quadratic equation whose roots
are $\mu_i,\nu_i$. Since these are real,
\[2(\mu_i^2+\nu_i^2)-(\mu_i+\nu+i^2)=(\mu_i-\nu_i)^2\ge0,\]
and so using equations (6.6), (6.7),
\[6p(\epsilon_i)p-1)-2\lambda_i)-2(\epsilon_i(p-1)-2\lambda_i)^2
-2(p_1)\lambda_i^2-(p-1)(\epsilon_i+\lambda_i)^2\ge0,\]
and this can be rearranged as given in the statement of the lemma.
\end{proof}

From Lemma~\ref{l106} we know that either one of the numbers $\epsilon_i$ is
$0$, or there are just three adjacency matrices and
$\epsilon_1=\epsilon_2=\epsilon_3=1$. We deal now with the former possibility;
notice that this must be the situation if we are dealing with case VIII.

\begin{lemma}
If $\epsilon_1=0$ then $p$ is $7$ or $19$, and the rank of our group is $4$.
If $p=7$ then $m_1=4$ and if $p=19$ then $m_1=6$.
\label{l109}
\end{lemma}

\begin{proof} If $\epsilon_1=0$ then $m_1=-2\lambda_1$ and $\lambda_1$ must
be negative. From Lemma~\ref{l108} we get
\[-12p\lambda_i-(3p+9)\lambda_i^2\ge0,\]
and so $4p+(p+3)\lambda_1\ge0$. Thus
$\lambda_1\ge(-4p)/(p+3)>-3$. 
\end{proof}

Consequently $\lambda_1$ is $-3$, $-2$ or $-1$, and $m_1$ is $6$, $4$ or $2$.
Since $G$ is primitive, $m_1$ cannot be $2$ (cf.~\cite{r20}, Theorem 18.7).
Hence $\lambda_1$ is $-3$ or $-2$. Now from Lemma~\ref{l107} $p-1$ divides $18$,
or $p-1$ divides $6$ respectively. In the former case $p$ is $19$ or $7$,
but if $p-7$ and $\lambda_1=-3$, returning to Lemma~\ref{l108} we see that the
necessary inequality is not satisfied. Therefore
\[\begin{array}{rccc}
&p=19,&\lambda_1=-3,&m_1=6\\[1mm]
\hbox{or }&p=7,&\lambda_1=-2,&m_1=4
\end{array}\]
are the only possibilities, and we are left with proving that the rank of our
group is $4$ and not $6$. To do this, put $A=\sum\{A_i\mid\epsilon_i=0\}$.
Then $A$ is a symmetric adjacency matrix of subdegree $m=\sum m_i$ and with
eigenvalue $\lambda=\sum\lambda_i$ of multiplicity $p+1$. The analysis which
we have already given applies as well to $A$ as to any of the matrices $A_i$.
Thus if $p=19$ then $\lambda=-3$, $m=6$; and if $p=7$ then $\lambda=-2$,
$m=4$. Consequently there can be no more than one summand, namely $A_1$,
going in to the making of $A$. So $\epsilon_i>0$ fof $i>1$. Since
$\sum\epsilon_i=3$ there can be at most $3$ further matrices $A_i$, and
reordering if necessary, the possiblities are $t=5$, 
$\epsilon_2=\epsilon_3=\epsilon_4=1$, or $t=3$ and $\epsilon_2=1$,
$\epsilon_3=2$. Remembering that, if $A_i$ arose from amalgamation of paired
colours, then $\epsilon_i$ is even, we see that in case VIII either $A_1$ or
both $A_1$ and $A_3$ must have arisen from amalgating paired colours. But if
$A_1$ comes from amalgamating paired colours then $\lambda_1$ is even. Only
$p=7$, $\lambda_1=-2$ satisfies this condition, but then the subdegree
corresponding to the original colour from which $A_1$ was obtained would be
$m_1/2=2$. This is impossible in a primitive group\footnote{[Footnote lost
from scan]} (\cite{r20}, 18.7), and this completes
the proof of Lemma~\ref{l109}. It is now quite easy to show that if $p=7$ then
the only possibilities for $\lambda_2,\lambda_3$ consistent with (6.9) and
with Lemma~\ref{l108} are $(1,0)$, $(0,1)$ or $(-1,2)$, and of these only
the last gives values of $\mu_1$, $\mu_2$, $\mu_3$, $\nu_1$, $\nu_2$,
$\nu_3$ which can satisfy (6.9). Thus for $p=7$ we get $n_1=4$, $n_2=8$,
$n_3=8$m and with suitable choice of notation,
\[\begin{array}{ll}
\mu_1=1+\sqrt{2} & \nu_1=1-\sqrt{2} \\[1mm]
\mu_2=-2\sqrt{2} & \nu_2=2\sqrt{2} \\[1mm]
\mu_3=-2+\sqrt{2} & \nu_3=-2-\sqrt{2}.
\end{array}\]

Similarly, if $p=19$ the possibilities for $\lambda_2$, $\lambda_3$
consistent with (6.9) and with Lemmas~\ref{l107}, \ref{l108} are $(3,-1)$,
$(2,0)$, $(0,2)$, $(-1,3)$. Only the last gives values of $\mu_1$, $\mu_2$,
$\mu_3$, $\nu_1$, $\nu_2$, $\nu_3$ which can satisfy (6.9), and so in this
case $n_1=6$, $n_2=20$, $n_3=30$, and with suitable choice of notation,
\[\begin{array}{ll}
\mu_1=\displaystyle{\frac{3+\sqrt{5}}{2}} & \nu_1=\displaystyle{\frac{3-\sqrt{5}}{2}}\\[3mm]
\mu_2=-2\sqrt{5} & \nu_2=2\sqrt{5}\\[1mm]
\mu_3=\displaystyle{\frac{-5+3\sqrt{5}}{2}} & \nu_3=\displaystyle{\frac{-5-3\sqrt{5}}{2}}.
\end{array}\]

Thus we have proved Theorem~\ref{t103}, and to complete the proof of
Theorem~\ref{t102} we need to deal with the case where the rank is $4$
and $\epsilon_1=\epsilon_2=\epsilon_3=1$. In this case the matrices $A_i$
are in fact just the basic adjacency matrices, $B_1$, $B_2$, $B_3$.

\begin{lemma}
If $\epsilon_1=\epsilon_2=\epsilon_3=1$ then $p=31$ and the subdegrees are
$1,20,32,40$.
\label{l1010}
\end{lemma}

The proof will be given as a series of numbered steps.

\begin{fact}
The numbers $\lambda_1,\lambda_2,\lambda_3$ are all different.
\end{fact}

For suppose not, $\lambda_1=\lambda_2$ say. Then, since $\epsilon_1=\epsilon_2$
we have $n_1=n_2$ and the quadratic equation derived from (6.6), (6.7) whose
roots are $\mu_1,\nu_1$ is the same as that whose roots are $\mu_2,\nu_2$.
Thus either $\mu_1=\mu_2$, $\nu_1=\nu_2$ or $\mu_1=\nu_2$, $\mu_2=\nu_1$.
Now since it is case VII which we are considering we know that $B_1$ and
$B_2$ are simultaneously diagonalisable. Therefore the argument used for
proving Lemma~\ref{l105} applies to complete the proof.

\begin{fact}
Put $a_i=3\lambda_i(\lambda_i+1)/(p-1)$ so that, by
Lemma~\ref{l107}, $a_i$ is an integer. Then if $\lambda_i\ge0$ we have
$a_i\le3$, and in any case $a_i\le4$.
\end{fact}

\begin{proof}
Since $\epsilon_i=1$, Lemma~\ref{l108} gives
\[(p-1)(4p+1)-6\lambda_i(p+1)-(3p+9)\lambda_i^2\ge0.\]
Hence
\begin{eqnarray*}
(3p+9)(\lambda_i^2+\lambda_i) &\le& (p-1)(4p+1)-(3p-3)\lambda_i,\\
\hbox{and }a_i=\frac{3\lambda_i(\lambda_i+1)}{p-1} &\le& 
\frac{4p+1}{p+3}-\frac{3\lambda_i}{p+3}\\
&=& 4-\frac{11}{p+3}-\frac{3\lambda_i}{p+3}.
\end{eqnarray*}
If $\lambda_i\ge0$ then we have $a_i<4$ and so $a_i\le3$. If $\lambda_i<0$
and $p\ge19$ then, as in the proof of Lemma~\ref{l106}, we have
$-3\lambda_i/(p+3)\le1$, hence $a_i<5$ and so $a_i\le4$. Finally, if $p<19$
and $\lambda_i\le0$ it is a simple matter to check that $a_i\le 3$ must hold.
\end{proof}

\begin{fact}
If none of $a_1,a_2,a_3$ is zero, then $a_1,a_2,a_3$ are all different.
\end{fact}

Suppose, for example, that $a_1=a_2$. Then both $\lambda_1,\lambda_2$ are roots
of the equation
\[3\lambda(\lambda+1)-a_1(p-1)=0.\]
Since (10.11) $\lambda_1\ne\lambda_2$ it follows that $\lambda_1+\lambda_2=-1$.
But from (6.9) we know that $\lambda_1+\lambda_2+\lambda_3=-1$, and so
$\lambda_3=0$. Then $a_3=0$ and (10.13) is proved.

\begin{fact}
If $a>0$ and $\lambda$ is a root of the equation
\[x^2+x-a=0\]
then $\lambda=-\frac{1}{2}\pm\sqrt{a}+\eta$, where $|\eta|<1/(8\sqrt{a})$.
\end{fact}

For $(\lambda+\frac{1}{2})^2=\lambda^2+\lambda+\frac{1}{4}=a+\frac{1}{4}$.

It is trivial to checl that the positive square root of $a+\frac{1}{4}$ lies
strictly between $\sqrt{a}$ and $\sqrt{a}+1/(8\sqrt{a})$, and that the
negative square root lies strictly between $-a-1/(8\sqrt{a})$ and $-\sqrt{a}$.

We shall now show that one of the numbers $a_i$ is zero by using the fact that
otherwise $p$ has three different representations in the form 
$1+3\lambda_i(\lambda_i+1)/a_i$ with three different (10.13) values of $a_i$,
where $a_i$ is $1$, $2$, $3$ or $4$ and $\lambda_1+\lambda_2+\lambda_3=-1$:

\begin{fact}
One of $a_1$, $a_2$, $a_3$ must be zero.
\end{fact}

\begin{proof} Suppose not: then (10.13) $a_1$, $a_2$, $a_3$ are all
different. We know that $\lambda_i$ is a root of the equation
\[x^2+x-a_i(p-1)/3=0\]
and so, by 10.14, we have
\[\lambda_i=-\frac{1}{2}\pm\sqrt{\frac{a_i(p-1)}{3}}+\eta_i\]
where
\[|\eta_i|<\frac{1}{8}\sqrt{\frac{3}{a_i(p-1)}}<\frac{1}{8}.\]
Now by (6.9) we have
\[\lambda_1+\lambda_2+\lambda_3=-1,\]
and so
\[-\frac{3}{2}+\sqrt{\frac{p-1}{3}}\left(\pm\sqrt{a_1}\pm\sqrt{a_2}
\pm\sqrt{a_3}\right)+\eta_1+\eta_2+\eta_3=-1.\]
This gives
\[\left|\sqrt{\frac{p-1}{3}}\left(\pm\sqrt{a_1}\pm\sqrt{a_2}\pm\sqrt{a_3}\right)
\right|<\frac{7}{8}.\]
We know (by 10.13 and 10.12) that $a_1$, $a_2$, $a_3$ are three among the
numbers $1,2,3,4$, and crude approximations to $\sqrt{2}$ and $\sqrt{3}$
give the estimate
\[\left|\pm\sqrt{a_1}\pm\sqrt{a_2}\pm\sqrt{a_3}\right| > \frac{4}{10}.\]
So we have
\[\frac{4}{10}\sqrt{\frac{p-1}{3}} < \frac{7}{8}.\]
This yields $p<15$: but elementary arithmetic is enough to show that no
prime less than $15$ has three representations in the form
$1+3\lambda_i(\lambda_i+1)/a_i$ with $\lambda_i,a_i$ integral and
$1\le a_i\le 4$. Thus our supposition leads to a contradiction, and one of
$a_1$, $a_2$, $a_3$ must be zero.
\end{proof}

Choose the notation so that $a_1=0$. Then

\begin{fact}
$\lambda_1=-1$.
\end{fact}

\begin{proof}
Certainly $\lambda_1$ is $0$ or $-1$. Assume therefore that $\lambda_1=0$. Not
both $a_2$, $a_3$ are zero, for otherwise $\lambda_1,\lambda_2,\lambda_3$ would
all be roots of the equation $x^2+x=0$ and they would not all be different as
we know (10.11) they must. Suppose therefore that $a_2\ne0$. Then
\[p=\frac{3\lambda_2(\lambda_2+1)}{a_2}+1.\]
If $3\nmid a_2$ then clearly $p\equiv1\pmod{3}$. If $3\mid a_2$ then by 10.12
$a_2=3$, $p=\lambda^2+\lambda+1$ and again we see that $p\equiv1\pmod{3}$.
Thus $p\equiv1\pmod{3}$ in all cases.

From equations (6.6), (6.7) we have
\begin{eqnarray*}
\mu_1+\nu_1 &=& -1,\\
\mu_1^2+\nu_1^2 &=& 3p=(p-1) = 2p+1,
\end{eqnarray*}
and (6.8) gives
\[3p(p-1)a_{111} = (p-1)^3+(p-1)(\mu_1^3+\nu_1^3).\]
Therefore
\begin{eqnarray*}
3pa_{111} &=& (p-1)^2 + \frac{3}{2}(\mu_1+\nu_1)(\mu_1^2+\nu_1^3)
-\frac{1}{2}(\mu_1+\nu_1)^3 \\
&=& (p-1)^2-\frac{3}{2}(2p+1)+\frac{1}{2}\\
&=& p^2-5p\\
3a_{111} &=& p-5.
\end{eqnarray*}
Since $a_{111}$ is an integer we see that $p\equiv5\pmod{3}$, which contradicts
the congruence we derived in the previous paragraph. Hence $\lambda_1\ne0$, so
$\lambda_1=-1$.

At last we are in a position to prove that $p=31$. Since $\lambda_1=-1$ and
$\lambda_1+\lambda_2+\lambda_3=-1$ we know that $\lambda_2=-\lambda_3=\lambda$, 
say. In this case
\begin{eqnarray*}
p-1 &\hbox{divides}&3\lambda(\lambda+1)\\
p-1 &\hbox{divides}&3(-\lambda)(-\lambda+1).
\end{eqnarray*}
Therefore $p-1$ divides $6\lambda$.

By interchanging $\lambda_2$, $\lambda_3$ if necessary we may assume that
$\lambda\ge0$. Then $\lambda\ne0$ since (10.11) $\lambda_2\ne\lambda_3$, and
$\lambda\ne1$ since $\lambda_3\ne\lambda_1$. Thus $\lambda\ge2$. Moreover, we
know from (10.12) that
\[\frac{6\lambda}{p-1}\cdot\frac{\lambda+1}{2}\le3.\]
Thus $\lambda+1\le6$ and if $t\lambda\ne p-1$ then $\lambda+1\le3$. It is now
easy to see that the only possibilities are
\[\begin{array}{rrr}
6\lambda=p-1, & \lambda=5, & p=31;\\[1mm]
\hbox{or} & \lambda=3, & p=19;\\[1mm]
\hbox{or\ \ } 3\lambda=p-1, & \lambda=2, & p=7.
\end{array}\]

In case $p=7$ we have $\lambda_2=2$ and $n_2=2$, but a primitive group of 
degree~$21$ cannot have a suborbit of length~$2$ (\cite[p51]{r20}). This case
therefore does not arise.

In case $p=19$ we use (6.6), (6.7) to calculate $\mu_1$, $\nu_1$, $\mu_2$,
$\nu_2$. $\mu_3$, $\nu_3$. The result is
\begin{eqnarray*}
\mu_1,\nu_1 &\hbox{are}& \pm2\sqrt{5} \\
\mu_2,\nu_2 &\hbox{are}& -2\pm\sqrt{6} \\
\mu_3,\nu_3 &\hbox{are}& 5, -3.
\end{eqnarray*}
Equation (6.9), $\mu_1+\mu_2+\mu_3=-1$ is now impossible to satisfy.

Finally, in case $p=31$, no such contradiction arises. With suitable
choices of roots one finds
\[\begin{array}{ccc}
\begin{array}{rcl}
\mu_1 &=& 4\sqrt{2} \\
\mu_2 &=& -3-\sqrt{2} \\
\mu_3 &=& 2-3\sqrt{2}
\end{array}
&\qquad&
\begin{array}{rcl}
\nu_1 &=& -4\sqrt{2} \\
\nu_2 &=& -3+\sqrt{2} \\
\nu_3 &=& 2+3\sqrt{2}.
\end{array}
\end{array}\]
And since $\lambda_1=-1$, $\lambda_2=5$, $\lambda_3=-5$, the subdegrees are
$n_1=32$, $n_2=20$, $n_3=40$. This is the third possibilitly in
Theorem~\ref{t102} and the proof is now complete.

\end{proof}

\begin{center}\Large\textbf{Part III}\\\textbf{Appendices}\end{center}

\section{A remark on imprimitive groups}

The analysis presented in \S\S4,5 does not apply to imprimitive groups in
general, but it is not hard to see that it is valid in case $p\ge5$ and $G$
is an insoluble group of degree~$3p$ which has no blocks of size $p$ 
in~$\Omega$. Such a group either is primitive, or, if imprimitive, then it has
$p$ blocks of size~$3$, and, by a well-known theorem of Burnside, the group
$G^*$ induced by $G$ on the set of blocks in $\Omega$ is $2$-fold transitive.
If the stabilizer of a block operates as $S_3$ on that block then $G$ is
cab=nonically embeddable in the monomial group (non-standard wreath product)
$S_3\wr G^*$; and if the stabilizer of a block operates as $A_3$ then $G$ is
canonically embeddable in $A_3\wr G^*$.

The irreducible character $\chi$ such that $1+\chi$ is the permutation
character of $G$ operating on the set of blocks in $\Omega$ is necessarily 
a constituent of our permutation character $\pi$ of degree~$3p$. Since $\chi$
has degree $p-1$, $G$ is of one of the Types III, IV, V, VII, VIII. It is not
hard to prove the facts listed below by methods similar (in most cases) to
those used in Part II of this paper. We shall use $G$ to denote an insoluble
group which is imprimitive of degree~$3p$, and which has no blocks of size~$p$.

\begin{fact}
If $G$ is of Type III then its subdegrees are $1,2,3(p-1)$.
\end{fact}

\medskip

Examples exist without restriction on $p$. For instance, if $G^*$ is any
insoluble group of degree~$p$ then $S_3\wr G^*$, as a group of degree~$3p$,
has Type~III. $\SL(3,3)$ operating as a group of degree~$39$ is also of this
kind.

\begin{fact}
If $G$ is of Type IV then either
\begin{itemize}
\item[(i)] its subdegrees are $1,2,p-1,2(p-1)$; or
\item[(ii)] its subdegrees are $1,1,1,3(p-1)$.
\end{itemize}
\end{fact}

In case (i) the stabilizer of a block operates on that block as the full
symmetric group $S_3$, and in case (ii) the stabilizer of a block operates
on it as $A_3$. Examples of case (i) occur whenever $p$ is a Fermat prime --
the smallest case being\footnote{This provides another, rather small, example
to settle the point raised in \cite{r20}, p.93, 11.3,4.}\label{p46}%
$S_5$ acting by
conjugation on its set of $15$ Sylow $2$-subgroups -- and I do not know
examples for other primes. Examples of case (ii) occur for all~$p$. For
example, if $G^*$ is any insoluble group of degree~$p$ then $A_3\wr G^*$ has
type IV and subdegrees $1,1,1,3(p-1)$.

\begin{fact}
If $G$ is of Type V then its subdegrees are $1,1,1,p-1,p-1,p-1$ and the
stabilizer of a block operates as $A_3$ on that block.
\end{fact}

Many examples exist. The groups $\SL(2,p-1)$, where $p$ is a Fermat prime,
operating as described in Lemma~\ref{l42}, are of this kind. Similarly, if
$q$ is a prime, $q\equiv1$~mod~$3$, and $p=(q^n-1)/(q-1)$ (this requires
$n\equiv\pm1$~mod~$6$). then $\PGL(n,q)$, $\PSL(n,q)$ can be faithfully 
represented as imprimitive groups of degree~$3p$ and Type V. By a theorem of
Ito~\cite{r11} these, the group $\SL(3,3)$ mentioned above, and $\PSL(2,7)$
are essentially the only groups $\PSL(n,q)$ of degree~$3p$.

\begin{fact}
There are no imprimitive groups of type VII.
\end{fact}

Thus Theorem~\ref{t102} is exhaustive for groups of this kind.

\begin{fact}
If $G$ is of type VIII then $p=7$, the subdegrees are
$1,2,2,4,4,8$, and $G$ is $\PSL(2,7)$ operating by conjugation on its set of
Sylow $2$-subgroups. 
\end{fact}

This is the group appearing in Lemma~\ref{l43} for $p=7$; it is the imprimitive
normal subgroup of index~$2$ in the group of Example~3 (p.~\pageref{e:ex3}).

\section{The index of $P$ in $N(P)$}

N. Ito, in \cite{r13}, has achieved a complete classification of transitive
simple permutation groups of degree $3p$ under the assumption that $N(P)$,
the normaliser of a Sylow $p$-subgroup $P$, has order $2p$. This raises the
question whether one can find a good bound for $N(P)$ in general. Of course,
there will be no significant such bound for $2$-fold transitive groups, but it
is easy to see that for primitive groups which are not $2$-fold transitive, 
the results of Part II imply that $|N(P):P|$ is always bounded by a function
whose order of magnitude\footnote{A similar deduction can be made from
Wielandt's results for simply transitive, primitive groups of degree~$2p$.}
is $4\sqrt{(3p)}$. Indeed, a little more precisely:

\begin{fact}
If $G$ is primitive of degree~$3p$ and Type III or IV,
then $|N(P):P|=q$ where $q$ divides one of $12$, $18$, $8a$, $12a$, 
$8(a+1)$ or $12(a+1)$.
\end{fact}

The proof is similar to the proof given below.

Curiously, for groups of Type II one gets far better information -- and in
view of Ito's work there is good hope that this is significant:

\begin{theorem}
If $G$ is of Type II and contains no odd permutations then $|N(P):P|$
divides~$8$. If $G$ does contain odd permutations then $|N(P):P|$ divides~$16$.
\end{theorem}

\begin{proof}
The second statement follows from the first since if $G$ contains odd
permutations then $G\cap A_{3p}$ is a subgroup of index~$2$ in $G$ which must
be primitive (cf.~Tamaschke's remark, 3.3 above; or compare~\S11), and of
Type~II. So let us suppose that $G$ contains no  odd permutations. Then
the centraliser $C(P)$ is $P$, for, if $C(P)$ contained a cycle of length
$3p$ then $G$ would be $2$-fold transitive by a theorem of Schur (see
\cite{r20},~p.65), and if $C(P)$ had order $2p$ then it would contain an
element of order~$2$ having $p$ fixed points and $p$ transpositions, an
odd permutation. Since $C(P)=P$ we know that $N(P)=P.Q$ where $Q$ is a
cyclic group of order $q$, say, $q$ divides $p-1$ and $Q$ acts faithfully
by conjugation as a group of automorphisms of $P$. A consequence of the
theorem of \S8 is that $p\equiv2$~mod~$3$, hence $3\nmid |N(P)|$ and $N(P)$
cannot be transitive on $\Omega$. Therefore $N(P)$ either has one orbit of
length $p$ and one of length $2p$, or three orbits of length $p$ in $\Omega$.
In either case $q$ is a stabiliser for an $N(P)$ orbit of length~$p$, and so
we may suppose that $Q\le H$. Either $Q$ has a further $2$ fixed points, or it
has an orbit of length~$2$; in both cases the remaining $3p-3$ points fall into
$Q$-orbits of length~$q$. A non-trivial $H$-orbt is a union of $Q$-orbits,
and it follows that
\[|\Gamma|\equiv 0, 1 \hbox{ or }2\hbox{ mod }q.\]
We know also that $p-1\equiv0$~mod~$q$. In case (i) (see \S8) we may take
$|\Gamma|=p+4a+1$, and we have\footnote{$(b,c)$ denotes the highest common
factor of $b$ and $c$.}
\[\begin{array}{lcl}
&&q\hbox{ divides }(48a^2+34a+6,48a^2+30a+4)\\[1mm]
\hbox{or}&\qquad&q\hbox{ divides }(48a^2+34a+5,48a^2+30a+4)\\[1mm]
\hbox{or}&\qquad&q\hbox{ divides }(48a^2+32a+4,48a^2+30a+4).
\end{array}\]
The Euclidean algorithm and a little computation yield that $q$ divides
$2$, $1$ and $(2a+4,8)$ respectively. That is, $q\mid 8$, and if $a$ is odd
then $q\mid2$. Similarly, in case (ii) we may take $|\Gamma|=p-4a-3$, and
we find that $q$ divides $2$, $1$ or $(8,2a+6)$. That is, $q\mid8$, and if 
$a$ is even thatn $q\mid2$. This proves the theorem.
\end{proof}

If $|N(P):P|=2$ it is not hard to show that $p=5$ and $G$ is $A_6$ as in
Example~1 (see also Ito~\cite{r13}). Thus we may in fact suppose that for
groups of type II(i) $a$ is even, and for groups of type II(ii) $a$ is odd.
However, I would guess that if $|N(P):P|=4$ then $G$ is $S_6$ as in
Example~1, and that groups of Type II with $|N(P):P|=8$ or~$16$ do not exist.

\section{A proof by Dr B. J. Birch}

The proof of Theorem~\ref{t92} in the case where all orbits are self-paired
eluded me for some time. Equations (9.7) and inequality (9.6) have two 
infinite families of solutions. One of these arises in case $\nu_1=-\nu_2=\nu$,
$\nu_3=-1$, in which case $p=\nu^2$, solutions of which are of no interest to
us since $p$ is a prime. The other is the family of solutions found in \S9,
where $p=3a^2+3a+1$, $\nu_1=2a+1$, $\nu_2=\nu_3=-a-1$. In fact, as Mr 
J.~E.~Stoy showed by computation, (9.6) and (9.7) do have other solutions,
but these appear to be relatively few and far between. At this point Dr Birch
proved for me that other soutions are indeed rare, and as his argument may
well extend for groups of degree $kp$ with $k>3$, I reproduce the relevant
part of his letter here.

``You consider the equations
\begin{equation}
\left.\begin{array}{rcl}
\gamma_1+\gamma+2+\gamma_3+1 &=& 0 \\[1mm]
\sqrt{4p-1-3\gamma_1^2}+\sqrt{4p-1-3\gamma_2^2} &=& \sqrt{4p-1-3\gamma_3^2}\\[1mm]
\hbox{with }4p-1 &\ge& 3\max(\gamma_i^2).\end{array}\right\}
\end{equation}
Removing square roots, one gets
\begin{equation}
\left.\begin{array}{rcl}
4p-1 &=& \gamma_1^2+\gamma_2^2+\gamma_3^2+2T \\
\hbox{with }T &\ge& 0
\end{array}\right\}
\end{equation}
and
\begin{equation}
\left.\begin{array}{c}
T^2=\gamma_1^4+\gamma_2^4+\gamma_3^4-\gamma_1^2\gamma_2^2-\gamma_2^2\gamma_3^2
-\gamma_3^2\gamma_1^2 \\
\gamma_1+\gamma_2+\gamma_3=0.
\end{array}\right\}
\end{equation}
Note that (3) is soluble with any two of $\gamma_1^2$, $\gamma_2^2$,
$\gamma_3^2$ equal -- leading to your two \emph{families} of solutions.

Suppose from now on that $|\gamma_1|>|\gamma_2|>|\gamma_3|>0$, so that
$\gamma_2$, $\gamma_3$, $\gamma_2-\gamma_3$, $\gamma_2+\gamma_3$ all have the
same sign (allowing $|\gamma_1|=|\gamma_2|$ or $|\gamma_2|=|\gamma_3|$ if
you like]. Write $\gamma_2+\gamma_3=S$, $\gamma_2-\gamma_3=D$. Then
\begin{eqnarray*}
4T^2 &=& \left(2\gamma_1^2-\gamma_2^2-\gamma_3^2\right)^2
+3\left(\gamma_2^3-\gamma_3^2\right)^2 \\
&=& \left(2(S+1)^2-\textstyle{\frac{1}{2}}S^2-\textstyle{\frac{1}{2}}D^2
\right)^2 +3S^2D_2.
\end{eqnarray*}
\[[2T+\{2(S+1)^2-\textstyle{\frac{1}{2}}(S^2+D^2)\}]
[2T-\{2(S+1)^2-\textstyle{\frac{1}{2}}(S^2+D^2)\}]=3S^2D^2\]
\begin{equation}
4(S+1)^2-S^2-D^2=P-Q\hbox{ with }PQ=3S^2D^2.
\end{equation}
Write $P=XY$ with $X\mid 3S^2$, $Y\mid D^2$; then $Q=ZT$ with $XZ=3S^2$,
$YT=D^2$, and we get
\begin{equation}
-8S-4=XZ-YT-XY+ZT=(X+T)(Z-Y).
\end{equation}
Take any large $N_0$ and small $\epsilon>0$. If now $p<N_0$, then by (2)
$S^2<\frac{3}{2}N_0$ (approximately). By (5) and $XZ=3S^2$, for fixed $S$
there are only $O(S^\epsilon)$ possibilities for $X,Y,Z,T$; so there are only
$O(N_0^{\frac{1}{2}+\epsilon})$ primes $p<N_0$ for which (1) is soluble.

In fact `sporadic' solutions of (4) appear to be extremely rare -- but I 
don't know how to prove this.'' 

\section{Tabulation of results}

The first of the following tables summarises the results of Part II. The
remaining ones give the eigenvalues of the basic adjacency matrices for every
case in which equations (6.6)--(6.9) are soluble. In these tables the columns
are indexed by the degrees of the irreducible constituents of $\pi$; the rows
are indexed by the basic adjacency matrices. The multiplication constants are
not tabulated; they can be calculated directly from the tables of eigenvalues
by means of (6.8).

\setcounter{table}{0}

\begin{landscape}

\begin{table}[htbp]
\[\begin{array}{|r|c|c|c|c|c|l|}
\hline
\hbox{Type} & p & & \hbox{rank} & \hbox{degrees of irreduc-} &
\hbox{subdegrees} & \hbox{Comments} \\
&&&& \hbox{ible constituents of $\pi$} & & \\[1mm]
\hline\hline
\hbox{VII (i)} & 7 & & & & 1,4,8,8 &
\hbox{$\PSL(2,7)$ -- Example 3} \\[1mm]
\hbox{(ii)} & 19 & & 4 & 1,p+1,p-1,p-1 &
1,6,20,30 & \hbox{$\PSL(2,19)$ -- Example 4} \\[1mm]
\hbox{(iii)} & 31 & & & & 1,20,32,40 & ? \\[1mm]
\hline
\hbox{II (i)} & 48a^2+30a+5 & a\ge0 & & &
1,2(8a+3)(3a+1),8(4a+1)(3a+1) & \hbox{$A_6,S_6$ -- Example 1, are} \\
& & \hbox{$a$ even} & & & & \hbox{examples for $a=0$} \\[1mm]
& & & 3 & 1,p,2p-1 & & \hbox{$N(P)$ is small, see \S12} \\[1mm]
\hbox{(ii)} & 48a^2+66a+23 & a\ge0 & & & 
1,2(8a+5)(3a+2),8(4a+3)(3a+2) & \\
& & \hbox{$a$ odd} & & & & \\[1mm]
\hline
\hbox{III (i)} & 3a^2+3a+1 & a\ge2 & & & 1,a(3a+1),2(a+1)(3a+1) &
\hbox{For $a=1$, $a_{111}=-1$, which} \\
& & & & & & \hbox{is not possible} \\[1mm]
& & & 3 & 1,2p,p-1 & & \\
\hbox{(ii)} & 3a^2+3a+1 & a\ge1 & & & 1,(a+1)(3a+2),2a(3a+2) &
\hbox{$A_7,S_7$ as in Example 2} \\
& & & & & & \hbox{arise for $a=1$} \\[1mm]
\hline
\hbox{IV (i)} & 3a^2+3a+1 & a\ge1 & & & 1,a(3a+1),(a+1)(3a+1),(a+1)(3a+1) &
\Delta_3=\Delta_2^*\hbox{\ \ I know no} \\
& & \hbox{$a$ even} & & & & \phantom{\Delta_3=\Delta_2^*}\hbox{\ \ examples of
any} \\
\hbox{(ii)} & 3a^2+3a+1 & a\ge1 & 4 & 1,p,p,p-1 & 1,(a+1)(3a+2),a(3a+2),a(3a+2) 
& \Delta_3=\Delta_2^*\hbox{\ \ of these kinds} \\
& & & & & & \phantom{\Delta_3=\Delta_2^*}\hbox{\ \ of group} \\[1mm]
\hbox{(iii)} & 3a^2+3a+1 & a\ge1 & & & 1,(a+1)(3a+2),a(3a+2),a(3a+2) & 
\hbox{All suborbits}\\
&&&&&& \hbox{are self-paired}\\[1mm]
\hline
\end{array}\]
\caption{\label{tab1}Summary of primitive groups of degree~$3p$}
\end{table}
\end{landscape}

\begin{table}[htbp]
\[\begin{array}{|c|cccc|}
\hline
B_i\setminus f_i & 1 & 8 & 6 & 6 \\[1mm]
\hline
B_0 & 1 & 1 & 1 & 1 \\[1mm]
\hline
B_1 & 4 & -2 & 1+\sqrt{2} & 1-\sqrt{2} \\[1mm]
\hline
B_2 & 8 & -1 & -2\sqrt{2} & 2\sqrt{2} \\[1mm]
\hline
B_3 & 8 & 2 & -2+\sqrt{2} & -2-\sqrt{2} \\[1mm]
\hline
\end{array}\]
\caption{\label{tab142}Type VII(i), $p=7$}
\end{table}

\begin{table}[htbp]
\[\begin{array}{|c|cccc|}
\hline
B_i\setminus f_i & 1 & 20 & 18 & 18 \\[1mm]
\hline
B_0 & 1 & 1 & 1 & 1 \\[1mm]
\hline
B_1 & 6 & -3 & \frac{1}{2}(3+\sqrt{5}) & \frac{1}{2}(3-\sqrt{5}) \\[1mm]
\hline
B_2 & 20 & -1 & -2\sqrt{5} & 2\sqrt{5} \\[1mm]
\hline
B_3 & 30 & 3 & \frac{1}{2}(-5+3\sqrt{5}) & \frac{1}{2}(-5-3\sqrt{5}) \\[1mm]
\hline
\end{array}\]
\caption{\label{tab143}Type VII(ii), $p=19$}
\end{table}

\begin{table}[htbp]
\[\begin{array}{|c|cccc|}
\hline
B_i\setminus f_i & 1 & 32 & 30 & 30 \\[1mm]
\hline
B_0 & 1 & 1 & 1 & 1 \\[1mm]
\hline
B_1 & 32 & -1 & 4\sqrt{2} & -4\sqrt{2} \\[1mm]
\hline
B_2 & 20 & 5 & -3-\sqrt{2} & -3+\sqrt{2} \\[1mm]
\hline
B_3 & 40 & -5 & 2-3\sqrt{2} & 2+3\sqrt{2} \\[1mm]
\hline
\end{array}\]
\caption{\label{tab144}Type VII(iii), $p=31$}
\end{table}

\clearpage

\begin{table}[htbp]
\[\begin{array}{|c|ccc|}
\hline
B_i\setminus f_i & 1 & 48a^2+30a+5 & 96a^2+60a+9 \\[1mm]
\hline
B_0 & 1 & 1 & 1 \\[1mm]
\hline
B_1 & 2(8a+3)(3a+1) & -8a-3 & 4a+1 \\[1mm]
\hline
B_2 & 8(4a+1)(3a+1) & 8a+2 & -4a-2 \\[1mm]
\hline
\end{array}\]
\caption{\label{tab145}Type II(i), $p=48a^2+30a+5$}
\end{table}

\begin{table}[htbp]
\[\begin{array}{|c|ccc|}
\hline
B_i\setminus f_i & 1 & 48a^2+66a+23 & 96a^2+132a+45 \\[1mm]
\hline
B_0 & 1 & 1 & 1 \\[1mm]
\hline
B_1 & 2(8a+5)(3a+2) & 8a+5 & -4a-3 \\[1mm]
\hline
B_2 & 8(4a+3)(3a+2) & -8a-6 & 4a+2 \\[1mm]
\hline
\end{array}\]
\caption{\label{tab146}Type II(ii), $p=48a^2+66a+23$}
\end{table}

\begin{table}[htbp]
\[\begin{array}{|c|ccc|}
\hline
B_i\setminus f_i & 1 & 6a^2+6a+2 & 3a^2+3a \\[1mm]
\hline
B_0 & 1 & 1 & 1 \\[1mm]
\hline
B_1 & a(3a+1) & a & -2a-1 \\[1mm]
\hline
B_2 & 2(a+1)(3a+1) & -a-1 & 2a \\[1mm]
\hline
\end{array}\]
\caption{\label{tab147}Type III(i), $p=3a^2+3a+1$}
\end{table}

\begin{table}[htbp]
\[\begin{array}{|c|ccc|}
\hline
B_i\setminus f_i & 1 & 6a^2+6a+2 & 3a^2+3a \\[1mm]
\hline 
B_0 & 1 & 1 & 1 \\[1mm]
\hline
B_1 & (a+1)(3a+2) & -a-1 & 2a+1 \\[1mm]
\hline
B_2 & 2a(3a+2) & a & -2a-2 \\[1mm]
\hline
\end{array}\]
\caption{\label{tab148}Type III(ii), $p=3a^2+3a+1$}
\end{table}

\begin{table}[htbp]
\[\begin{array}{|c|cccc|}
\hline
B_i\setminus f_i & 1 & 3a^2+3a+1 & 3a^2+3a+1 & 3a^2+3a \\[1mm]
\hline
B_0 & 1 & 1 & 1 & 1 \\[1mm]
\hline
B_1 & a(3a+1) & a & a & -2a-1 \\[1mm]
\hline
B_2 & (a+1)(3a+1) & \alpha & \beta & a \\[1mm]
\hline
B_3 & (a+1)(3a+1) & \beta & \alpha & a \\[1mm]
\hline
\end{array}\]
\caption{\label{tab149}Type IV(i), $p=3a^2+3a+1$, $a$ even\\
$\alpha+\beta=-a-1$\qquad$\alpha\beta=\frac{1}{2}(5a+2)(a+1)$}
\end{table}

\begin{table}[htbp]
\[\begin{array}{|c|cccc|}
\hline
B_i\setminus f_i & 1 & 3a^2+3a+1 & 3a^2+3a+1 & 3a^2+3a \\[1mm]
\hline
B_0 & 1 & 1 & 1 & 1 \\[1mm]
\hline
B_1 & (a+1)(3a+2) & -a-1 & -a-1 & 2a+1 \\[1mm]
\hline
B_2 & a(3a+2) & \alpha & \beta & -a-1 \\[1mm]
\hline
B_3 & a(3a+2) & \beta & \alpha & -a-1 \\[1mm]
\hline
\end{array}\]
\caption{\label{tab1410}Type IV(ii),(iii), $p=3a^2+3a+1$\\
Type IV(ii): $a$ is odd and $\alpha+\beta=a$, $\alpha\beta=\frac{1}{2}a(5a+3)$\\
Type IV(iii): $a$ is even and $\alpha+\beta=a$, $\alpha\beta=-\frac{1}{2}a(4a+3)$}
\end{table}

\bigskip\bigskip\noindent
The Queen's College, Oxford\\
January, 1969

\bigskip\bigskip\noindent
\hrule

\subsection*{Afterword}

By Peter J. Cameron, University of St Andrews

As said above, this paper was written in the academic year 1968--69. For 
various reasons, including the fact that other mathematicians (Olaf
Tamaschke and Leonard Scott) had proved overlapping results, it was never
published. I would like to explain why I am making it available now.
There are three main reasons.

\paragraph{Historical significance} The paper was written at a time when 
mathematical objects from three different areas (association schemes in
statistics, coherent configurations in permutation groups, and cellular
algebras in computer science), were being recognised as essentially the same
thing, and to provide a framework for using combinatorial methods in the
study of permutation groups. It gives a very good introduction to this
material and would have been very influential had it been published at the
time. As it was, its effect was indirect: for example, I read it as a doctoral
student, and the ideas were present in my thesis.

\paragraph{Mathematical content} The arguments are much more difficult and
subtle than those used by Wielandt, and the mathematics there which could be
valuable in other contexts.

\paragraph{Mathematical significance} The main theorem of the paper has been
subsumed by results proved using the Classification of Finite Simple
Groups (which was still in the future when the paper was written): for
example, we now have a classification of primitive groups of odd degree%
\footnote{W. M. Kantor, Primitive permutation groups of odd degree, and an
application to finite projective planes, \textit{J. Algebra} \textbf{106}
(1987), 15--45; Martin W. Liebeck and Jan Saxl, The primitive permutation
groups of odd degree, \textit{J. London Math. Soc.} (2) \textbf{31} (1985),
250--264.}. But it has other implications which have not been drawn out. It is
known that Wielandt's arguments can be used to prove a purely combinatorial 
result, concerning strongly regular graphs with prescribed eigenvalue
multiplicities\footnote{P. J. Cameron and J. H. van Lint, \textit{Designs,
Graphs, Codes and their Links}, London Math. Soc. Student Texts \textbf{22},
Cambeidge University Press, Cambridge, 1991, Theorem~2.20.}, and it is likely
that those in this paper will have similar implications for strongly regular
graphs and, more generally, for association schemes. So new mathematics will
result from the paper, and it will be important to have it available to
researchers.

\medskip

As a fourth, more personal reason, I want it to be a tribute to
Peter Neumann, an important mathematician and historian of mathematics,
an inspiring teacher, and a fine and generous person. I was his sixth
doctoral student and regard myself as in his debt.

\medskip

I have re-typed the paper in \LaTeX\ from a scan of the original typescript. In
a couple of places the text was illegible; specifically footnote 2 on page~12
of the present document and the surrounding text (where I have made a guess as
to what was written) and footnote 5 on page~28 (which was cut off in the scan
but seems to me to be unnecessary).

\medskip

I am grateful to Leonard Scott for providing me with a scan of the paper and
of related historical material on his proposed collaboration with Peter
Neumann (which did not materialise), to Sylvia Neumann for giving permission
to post this paper, and to Cheryl E. Praeger and others for encouraging me to
do so.

\paragraph{Final note} This second arXiv version corrects some typographic
errors (both by Neumann in the original and by me while re-typing). I am
very grateful to Marina Anagnostopoulou-Merkouri for spotting these.
\end{document}